\documentclass[12pt,leqno]{amsart}
\usepackage{latexsym}
\usepackage{amsmath}
\usepackage{amsthm}
\usepackage{amssymb}
\usepackage{amsfonts}

\def\curl{\operatorname{curl}}
\usepackage{mathrsfs}
\usepackage{fullpage}
\usepackage{graphicx}
\begin{document}
\theoremstyle{plain}
\newtheorem{theorem}{Theorem}[section]
\newtheorem{MainThm}{Theorem}
\newtheorem{thm}{Theorem}[section]
\newtheorem{clry}[thm]{Corollary}
\newtheorem{notation}[]{Notation}
\newtheorem{proposition}[theorem]{Proposition}
\newtheorem{lemma}[thm]{Lemma}
\newtheorem{deft}[thm]{Definition}
\newtheorem{hyp}{Assumption}
\newtheorem*{conjecture}{Conjecture}
\newtheorem{question}{Question}
\newtheorem{Case}{Case}
\newtheorem{SCase}{Case}
\newtheorem{SSCase}{Case}
\newtheorem{SSSCase}{Case}
\newtheorem{SSSSCase}{Case}
\newtheorem{Step}{Step}
\newtheorem{SStep}{Step}
\newtheorem{remark}[thm]{Remark}
\newtheorem{claim}[thm]{Claim}

\theoremstyle{definition}
\newtheorem*{definition}{Definition}
\newtheorem{assumption}{Assumption}
\newtheorem{rem}[thm]{Remark}
\newtheorem*{acknow}{Acknowledgements}

\newtheorem{example}[thm]{Example}
\newtheorem*{examplenonum}{Example}
\numberwithin{equation}{section}
\newcommand{\nocontentsline}[3]{}
\newcommand{\tocless}[2]{\bgroup\let\addcontentsline=\nocontentsline#1{#2}\egroup}
\newcommand{\eps}{\varepsilon}
\renewcommand{\phi}{\varphi}
\renewcommand{\d}{\partial}
\newcommand{\re}{\mathop{\rm Re} }
\newcommand{\im}{\mathop{\rm Im}}
\newcommand{\mR}{\mathbb{R}}
\newcommand{\mC}{\mathbb{C}}
\newcommand{\mN}{\mathbb{N}} 
\newcommand{\mZ}{\mathbb{Z}} 
\newcommand{\mK}{\mathbb{K}}
\newcommand{\supp}{\mathop{\rm supp}}
\newcommand{\abs}[1]{\lvert #1 \rvert}
\newcommand{\norm}[1]{\lVert #1 \rVert}
\newcommand{\csubset}{\Subset}
\newcommand{\detg}{\lvert g \rvert}
\newcommand{\msetminus}{\setminus}

\newcommand{\br}[1]{\langle #1 \rangle}

\newcommand{\ehat}{\,\hat{\rule{0pt}{6pt}}\,}
\newcommand{\echeck}{\,\check{\rule{0pt}{6pt}}\,}
\newcommand{\etilde}{\,\tilde{\rule{0pt}{6pt}}\,}

\newcommand{\tr}{\mathrm{tr}}
\newcommand{\mdiv}{\mathrm{div}}

\allowdisplaybreaks

\title{An $H^{s,p}(\curl;\Omega)$ estimate for the Maxwell system}

\author{Manas Kar}
\address{Department of Mathematics and Statistics,
University of Jyv{\"a}skyl{\"a},
P.O.Box 35 (MaD) FI-40014, Jyv{\"a}skyl{\"a}, Finland.}
\email{manas.m.kar@maths.jyu.fi}

\author{Mourad Sini}
\address{RICAM, Austrian Academy of Sciences,
Altenbergerstrasse 69, A-4040, Linz, Austria.}
\email{mourad.sini@oeaw.ac.at}

\date{\today}
%\maketitle

\begin{abstract}
We derive an $H_{0}^{s,p}(\curl;\Omega)$ estimate for the solutions of the Maxwell type equations modeled with anisotropic and $W^{s, \infty}(\Omega)$-regular coefficients.  
Here, we obtain the regularity of the solutions for the integrability and smoothness indices $(p, s)$ in a plan domain characterized by the apriori lower/upper bounds of $a$ 
and the apriori upper bound of its H{\"o}lder semi-norm of order $s$. The proof relies on a perturbation argument generalizing Gr{\"o}ger's $L^p$-type estimate, known for the elliptic problems, to the Maxwell system.
\end{abstract}

\maketitle

%\begin{keywords} 
%Maxwell system, Sobolev-Besov space,well posedness, complex interpolation, fixed point theorem, contraction map, perturbation argument.
%\end{keywords}
%
%\begin{AMS}
%35J25,78A25, 35Q61, 35D30
%\end{AMS}

%\tableofcontents
\section{Introduction}
Assume $\Omega\subset\mathbb{R}^3$ to be a bounded and Lipschitz domain. 
Let the coefficient $a$ be a $3 \times 3$ matrix, with elements in $W^{s, \infty}(\Omega)$, $s\geq 0$, satisfying the uniform ellipticity condition, i.e.,
there exist positive constants $m,  M$ such that
\begin{equation}\label{ellipticity}
m |\xi|^2 \leq a(x)\xi\cdot\overline{\xi} \leq M |\xi|^2,
\end{equation}
for all $\xi\in\mathbb{C}^3$ and almost every $x\in\Omega$ and having, if $s>0$, a bounded H{\"o}lder semi-norm with exponent $s$, i.e., there exists $\tilde M>0$ such that 
\begin{equation}\label{Holder-bound}
|a|_{C^{0,s}}\leq \tilde M.
\end{equation}
The goal of this work is to 
study the well posedness of the following boundary value problem
\begin{equation}\label{max}
\begin{cases}
\curl (a\curl u) + k^2u = f, \ \text{in}\ \Omega \\
\nu\wedge u = 0, \ \text{on}\ \partial\Omega
\end{cases}
\end{equation}
in the appropriate Sobolev spaces with fractional order.
Here, we denote by $\nu$ the outer unit normal on $\partial\Omega$ and $k$ the frequency.
The problem \eqref{max} covers the case \footnote{One way to deal with the general case where $k^2$ is replaced by $b \in (L^{\infty}(\Omega))^3$ 
(for example in (\ref{maxE}) replace $k^2$ by  $b(x):=-k^2 \epsilon^{-1}(x)$, $x\in \Omega$, with  $\epsilon \in (L^\infty(\Omega))^{3\times 3}$ 
lower bounded by a positive constant) is discussed in Remark \ref{general-case}.} when the electric field $E$ satisfies
\begin{equation}\label{maxE}
\begin{cases}
\curl (\mu^{-1}\curl E) + k^2E = f, \ \text{in}\ \Omega \\
\nu\wedge E = 0, \ \text{on}\ \partial\Omega
\end{cases}
\end{equation}
 or the magnetic field $H$ satisfies
\begin{equation}\label{maxH}
\begin{cases}
\curl (\epsilon^{-1}\curl H) + k^2H = f, \ \text{in}\ \Omega \\
\nu\wedge H = 0, \ \text{on}\ \partial\Omega
\end{cases}
\end{equation}
with anisotropic permitivity $\epsilon$ and permiability $\mu$.
\bigskip
 
The well posedness of the problem \eqref{max} has been derived in the $L^2$-based Sobolev spaces, i.e., $H(\curl;\Omega)$,
for domains with minimal smoothness and minimum regularity assumptions on the coefficients, see for instance \cite{Monk}.  
\bigskip

In order to study the question of regularity for the inhomogeneous Dirichlet-Laplace problem, in \cite{Je-Keni}, Jerison and Kenig used harmonic analysis 
technique to obtain a best possible estimates for the solutions in Sobolev-Besov $L_{s}^{p}(\Omega)$ norms with optimal range 
of the smoothness index $s$ and the integrability index $p$. In \cite{Mitrea,MiT}, M. Mitrea, D. Mitrea and J. Pipher considered an inhomogeneous Maxwell equations 
in a Lipschitz sub-domain, where the electric permitivity $\epsilon$ and the magnetic permeability $\mu$ taken to be constants. 
The regularity estimate for the solutions of \eqref{max}, developed in the $L^p$ settings for the optimal values of $p$'s,
can be found in \cite{Mitrea}. In \cite{MiT}, M. Mitrea showed the well posedness in the Sobolev-Besov spaces $H^{s, p}_0(\curl;\Omega)$ with
the optimal range of the smoothness index $s$ and the integrability index $p$ which generalizes, to the Maxwell system, the results by Jerison and Kenig
mentioned above.
\bigskip
 
Regarding variable coefficients and under weak regularity assumptions on the domain $\Omega$ and only $L^\infty$-regularity assumption of the coefficients, the well posedness for the 
divergence form elliptic problems has been studied by Gr{\"o}ger. In \cite{Groe}, he demonstrated the well posedness in the Sobolev space 
$W^{1,p}(\Omega)$ for $p>2$, which is a generalization of the work of Meyers \cite{Me}, known for Dirichlet boundary conditions, to mixed type boundary conditions.
The proof is based on a perturbation argument via the Banach fixed point theorem, see \cite{Groe}. 
We refer the reader to the text books \cite{Ev, GT1} for $L^p$-estimates of the solutions of elliptic problems in case of smooth coefficients.
For the Maxwell model, 
a $W^{1,p}$-type regularity estimate for the solutions
has been derived by Bao, Li, and Zhou considering $\mu$ to be constant and $\epsilon$ as piecewise constant, see \cite{BYZ}. 
Related estimates for smooth coefficients are derived  by Yin in \cite{Yin}, see also the references therein.
In the recent work \cite{K-S-2}, we proved an estimate of the solutions for the
problem in the Sobolev spaces $H_{0}^{1,p}(\curl;\Omega)$, for $p$ near $2$, where $a$ is taken to be a matrix valued function 
satisfying (\ref{ellipticity}). 
\bigskip

In this work, we use the approach by Gr{\"o}ger to deal with the regularity issue regarding the model \eqref{max} in the spaces $H^{s, p}_0(\curl;\Omega)$ for 
$W^{s, \infty}(\Omega)$ coefficients for a certain range of $s$ and $p$, see Theorem \ref{pert} and Figure \ref{Fig 1}. This completes the work in \cite{K-S-2} and provides the 
 Gr{\"o}ger-Meyers's regularity estimate corresponding to the formentioned results in \cite{MiT}.
Let us also mention that compared to the works in \cite{BYZ} and \cite{Yin},  
our estimates are derived for less regular coefficients (for instance for $s=0$) and show that the solution operator for the model \eqref{max} is an isomorphism. 
In addition to the general interest of such regularity estimates, this isomorphism property is useful for justifying a shape reconstruction algorithm in the theory of inverse problems,
see \cite{K-S-2, SiY}.
\bigskip

The paper is organized as follows. In Section \ref{defi}, we recall the basic definitions of the
Sobolev and Besov spaces of functions in Lipschitz domains and also some functional properties on those spaces. 
Then we state the main result in this paper in Section \ref{theosec} and finally, a detailed proof of the result is given in Section \ref{proofsec}. 

\begin{section}{Definitions and preliminary results}\label{defi}
\subsection{Sobolev and Besov spaces in Lipschitz domains}
For $1<p<\infty$ and $-\infty<s<\infty$ the Sobolev space $L_{s}^{p}(\mathbb{R}^3)$ is defined by \footnote{The space $L_{s}^{p}(\mathbb{R}^3)$
can also be defined using the Fourier transform 
$
 L_{s}^{p}(\mathbb{R}^3) := \{f ; f\in\mathcal{S}', \|f\|_{p}^{s} < \infty\},
$
with the norm 
$
 \|f\|_{p}^{s} = \|\mathcal{F}^{-1}\{(1+|\xi|^2)^{\frac{s}{2}}\mathcal{F}{f}\}\|_{L^p(\mathbb{R}^3)},
$
where $s\in\mathbb{R}$. Here, $\mathcal{F}$ and $\mathcal{F}^{-1}$ represent the Fourier transform and inverse Fourier transform respectively and 
$\mathcal{S}'$ represents the space of tempered distributions.}
\[
L_{s}^{p}(\mathbb{R}^3) := \{(I-\Delta)^{-\frac{s}{2}}g; g\in L^p(\mathbb{R}^3)\},
\]
 with the norm
\[
 \|f\|_{L_{s}^{p}(\mathbb{R}^3)} = \|(I-\Delta)^{\frac{s}{2}}f\|_{L^p(\mathbb{R}^3)}.
\]

 For any Lipschitz domain $\Omega\subset\mathbb{R}^3$, 
we denote by $L_{s}^{p}(\Omega)$ the Sobolev space defined as the restrictions to $\Omega$ of the elements in $L_{s}^{p}(\mathbb{R}^3).$
The norm is defined 
as follows:
\[
\|f\|_{L_{s}^{p}(\Omega)} := \inf \{\|g\|_{L_{s}^{p}(\mathbb{R}^3)} ; \mathcal{R}_{\Omega}g =f \},
\]
where $\mathcal{R}_{\Omega}g$
denotes the restriction of the function $g$ from $\mathbb{R}^3$ to $\Omega$.
Moreover, if we define the space $W^{m,p}(\Omega)$ by 
\[
 W^{m,p}(\Omega) :=\{f\in L^p(\Omega) ; \ \frac{\partial^{\beta}f}{\partial x^{\beta}}\in L^p(\Omega), |\beta|\leq m\},
\]
equipped with the norm
\[
 \|f\|_{W^{m,p}(\Omega)} := \left(\sum_{|\beta|\leq m}\int_{\Omega}|\frac{\partial^{\beta}f}{\partial x^{\beta}}|^pdx\right)^{1/p},
\]
for $m\geq 1$, an integer, and $1<p<\infty$, then we have the equality
\[
 L_{m}^{p}(\Omega) = W^{m,p}(\Omega),
\]
see for instance \cite{Bergh} and \cite{Je-Keni}, where as usual $\frac{\partial^{\beta}}{\partial x^{\beta}} := 
\frac{\partial^{|\beta|}}{\partial x_{1}^{\beta_1}\partial x_{2}^{\beta_2}\partial x_{3}^{\beta_3}}$ with $|\beta| = \beta_1+\beta_2+\beta_3.$
Using Stein's extension operator, the space $L_{s}^{p}(\Omega)$, $0<s<1$, can be interpreted as the complex interpolation \footnote{A detailed discussion about this space can be found in \cite{Adams} and \cite{Bergh}, for instance.} space between $L^p(\Omega)$ and $W^{1,p}(\Omega)$, i.e.,
\[
[L^p(\Omega),W^{1,p}(\Omega)]_{[s]} = L_{s}^{p}(\Omega), 
\]
for all $1<p<\infty$.
For $p=\infty$ and $0<s<1$, the space $W^{s,\infty}(\Omega)$ can be viewed as the space of functions
\[
\left\{ \phi\in L^{\infty}(\Omega) : \frac{|\phi(x)-\phi(y)|}{|x-y|^s} \in L^{\infty}(\Omega\times\Omega)\right\}.
\]
Basically this space is equivalent to the H{\"o}lder continuous space $C^{0,s}(\Omega)$ with exponent $s$ and the norm can be defined as
\[
\|\phi\|_{W^{s,\infty}(\Omega)}
= \|\phi\|_{L^{\infty}(\Omega)} + |\phi|_{C^{0,s}},
\]
where the H{\"o}lder semi-norm is denoted by
\[
|\phi|_{C^{0,s}} := \sup_{\substack{
x, y\in\Omega \\
x\neq y}}
\frac{|\phi(x)-\phi(y)|}{|x-y|^s}.
 \]

According to \cite{Je-Keni}, for $1<p<\infty$ and $s\in\mathbb{R},$ we define $L_{s,0}^{p}(\Omega)$
as the space of all distributions $f \in L_{s}^{p}(\mathbb{R}^3)$ such that 
$supp \ f \subset\overline{\Omega}$
and the norm is 
$$ \|f\|_{L_{s,0}^{p}(\Omega)} := \|f\|_{L_{s}^{p}(\mathbb{R}^3)}.$$
It is known that $C_{0}^{\infty}(\Omega)$ is dense in $L_{s,0}^{p}(\Omega)$ for all values of $s$ and $p $ with $p>1$.
For positive $s$, $L_{-s}^{p}(\Omega)$ is defined as the space of distributions in $\Omega$ such that
\[
 \|f\|_{L_{-s}^{p}(\Omega)} := \sup\{|\langle f, \varphi\rangle|; \ \varphi\in C_{0}^{\infty}(\Omega), \|\tilde\varphi\|_{L_{s}^{q}(\mathbb{R}^3)}\leq 1\}<\infty,
\]
where tilde denotes the extension by zero outside $\Omega$ and $1/p+ 1/q = 1.$ For all values of $p$ and $s$, $C^{\infty}(\overline{\Omega})$
is dense in $L_{s}^{p}(\Omega).$ Also, $C_{0}^{\infty}(\Omega)$ is dense in $L_{s}^{p}(\Omega)$, for $s\leq0$. In addition, for any $s\in\mathbb{R},$
\[
 L_{-s,0}^{q}(\Omega) = (L_{s}^{p}(\Omega))^{'} \ \ \text{and}\ \ L_{-s}^{q}(\Omega) = (L_{s,0}^{p}(\Omega))^{'},
\]
see for instance in [\cite{Je-Keni}, Proposition 2.4, Proposition 2.9] and \cite{MiT}.
For each $p$ and $s$ satisfying $1<p<\infty,$ $-1+1/p<s<1/p,$ there exists a linear and bounded extension operator 
\[
E_{ext} : L_{s}^{p}(\Omega) \rightarrow  L_{s}^{p}(\mathbb{R}^3),
\]
 by
\[
E_{ext}(u) = \tilde{u},
\]
with the property that $supp \ \tilde u\subset\overline{\Omega}$, see [\cite{HTRi}, Theorem 3.5].
In addition, $Range (E_{ext}) = L_{s,0}^{p}(\Omega)$, which allows the following identification
\[
 L_{s}^{p}(\Omega) = L_{s,0}^{p}(\Omega); \ \ \forall \ p\in(1,\infty), \ \forall \ s\in(-1+1/p,1/p).
\]
Thus, if $p,q\in(1,\infty)$ are such that $ 1 / p+ 1 / q=1,$ then
\begin{equation}\label{dual}
 (L_{s}^{p}(\Omega))^{'} = L_{-s}^{q}(\Omega), \ \ \forall\ s\in(-1+{1/p},{1/p}),
\end{equation}
and hence $L_{s}^{p}(\Omega)$ is reflexive. Note that the product space
$L_{s}^{p}(\Omega)\times L_{s}^{p}(\Omega)$ is a Banach space with the usual graph norm as well as with the equivalent norm
\[
\|(f_1,f_2)\|_{L_{s}^{p}(\Omega)\times L_{s}^{p}(\Omega)} = \left(\|f_1\|_{L_{s}^{p}(\Omega)}^{p} + \|f_2\|_{L_{s}^{p}(\Omega)}^{p}\right)^{{1/p}}. 
\]

In the following lemma we discuss the characterization of the dual of $L_{s}^{p}(\Omega)\times L_{s}^{p}(\Omega)$ with suitable $p$ and $s$.
\begin{lemma}\label{Represation theo}
Let $1<p, q<\infty$ be real numbers with ${1/p} + {1/q} = 1$ and $s$ be such that $-1+{1/p}<s<{1/p}$. Assume that $\Omega\subset\mathbb{R}^3$ be a bounded Lipschitz domain.
Then for every fixed $g = (g_1,g_2) \in L_{s}^{p}(\Omega) \times L_{s}^{p}(\Omega),$ the operator $\mathcal{F}^*$
defined by
\[
 \mathcal{F}^*f := {}_{L_{s}^{p}(\mathbb{R}^3)}\langle \tilde{g_1}, \tilde{f_1}\rangle_{L_{-s}^{q}(\mathbb{R}^3)} + {}_{L_{s}^{p}(\mathbb{R}^3)}\langle \tilde{g_2}, \tilde{f_2} \rangle_{L_{-s}^{q}(\mathbb{R}^3)},
\]
is a continuous linear functional on $L_{-s}^{q}(\Omega) \times L_{-s}^{q}(\Omega),$
where $\tilde{f_j}\in L_{-s}^{q}(\mathbb{R}^3)$ is any extension of $f_j$ satisfying $\mathcal{R}_{\Omega}\tilde{f_j} = f_j$ and $\tilde g_j$ is an extension of $g_j$ by zero outside $\Omega$, for $j=1,2$.  

Conversely, every $\mathcal{F}^* \in (L_{-s}^{q}(\Omega)\times L_{-s}^{q}(\Omega))^{'}$ can be written in the above form with a uniquely determined $g \in L_{s}^{p}(\Omega)\times L_{s}^{p}(\Omega).$
Moreover, there exist $c_1, c_2>0$ such that
\[
 c_1 \|g\|_{L_{s}^{p}(\Omega) \times L_{s}^{p}(\Omega)}\leq \|\mathcal{F}^*\|_{(L_{-s}^{q}(\Omega)\times L_{-s}^{q}(\Omega))^{'}} \leq c_2 \|g\|_{L_{s}^{p}(\Omega) \times L_{s}^{p}(\Omega)}.
\]
\end{lemma}
\begin{proof}
Applying the H{\"o}lder inequality we have the continuity of $\mathcal{F}^*$ together with the estimate
\[
 \|\mathcal{F}^*\|_{(L_{-s}^{q}(\Omega)\times L_{-s}^{q}(\Omega))^{'}} \leq c_2 \|g\|_{L_{s}^{p}(\Omega) \times L_{s}^{p}(\Omega)}.
\]
 Conversely, for $j=1,2$, we define $$T_j : L_{-s}^{q}(\Omega) \rightarrow L_{-s}^{q}(\Omega) \times L_{-s}^{q}(\Omega)$$  by
\[
 T_1f_1 := (f_1,0) 
\]
and 
\[
 T_2f_2 := (0,f_2),
\]
for all $ f_1,f_2 \in L_{-s}^{q}(\Omega).$ Note that
$T_j$ is a continuous and linear operator. Since $\mathcal{F}^*\in (L_{-s}^{q}(\Omega)\times L_{-s}^{q}(\Omega))^{'}$, then for 
$f = (f_1,f_2) \in L_{-s}^{q}(\Omega) \times L_{-s}^{q}(\Omega),$ we have
\[
 \mathcal{F}^*f = (\mathcal{F}^*\circ T_1)f_1 + (\mathcal{F}^*\circ T_2)f_2.
\]
Since $T_j$ and $\mathcal{F}^*$ are linear and continuous then $\mathcal{F}^*\circ T_j \in (L_{-s}^{q}(\Omega))^{'},$ for $j=1,2.$
Note that, from \eqref{dual} we have the characterization of the dual space 
$(L_{-s}^{q}(\Omega))^{'}$, i.e., $$(L_{-s}^{q}(\Omega))^{'} = L_{s}^{p}(\Omega),$$
for all $-1+1/p<s<1/p$.
Therefore, there exists a unique $g_j\in L_{s}^{p}(\Omega)$ such that
\[
 (\mathcal{F}^*\circ T_j)(f_j) = {}_{L_{s}^{p}(\mathbb{R}^3)}\langle \tilde{g_j},\tilde{f_j} \rangle_{L_{-s}^{q}(\mathbb{R}^3)}, 
\]
where $\tilde{f_j}\in L_{-s}^{q}(\mathbb{R}^3)$ is any extension of $f_j$ satisfying $\mathcal{R}_{\Omega}\tilde{f_j} = f_j$ and $\tilde g$ is an extension of $g$ by zero outside $\Omega$ for $j= 1,2$. 
Now, define $g := (g_1,g_2) \in L_{s}^{p}(\Omega) \times L_{s}^{p}(\Omega).$ Therefore,
\[
 \mathcal{F}^*f = {}_{L_{s}^{p}(\mathbb{R}^3)}\langle \tilde{g_1},\tilde{f_1}\rangle_{L_{-s}^{q}(\mathbb{R}^3)} + {}_{L_{s}^{p}(\mathbb{R}^3)}\langle \tilde{g_2}, \tilde{f_2}\rangle_{L_{-s}^{q}(\mathbb{R}^3)}.
\]
Thus the natural mapping $$A : L_{s}^{p}(\Omega) \times L_{s}^{p}(\Omega) \rightarrow (L_{-s}^{q}(\Omega) \times L_{-s}^{q}(\Omega))^{'},$$ defined by
$$(g_1,g_2) \longmapsto \mathcal{F}^*,$$ is bounded linear and bijective. Hence by the open mapping theorem we conclude that $A$ is an isomorphism, i.e., there
exist $c_1,c_2>0$ such that
\[
 c_1 \|g\|_{L_{s}^{p}(\Omega) \times L_{s}^{p}(\Omega)}\leq \|\mathcal{F}^*\|_{(L_{-s}^{q}(\Omega)\times L_{-s}^{q}(\Omega))^{'}} \leq c_2 \|g\|_{L_{s}^{p}(\Omega) \times L_{s}^{p}(\Omega)}.
\]
\end{proof}
The following lemma describes the interpolation spaces by applying the complex interpolation method.
\begin{lemma}\label{InTERp}
 We have the following characterization
\[
 L_{s}^{p}(\Omega) = [L_{s_0}^{p_0}(\Omega), L_{s_1}^{p_1}(\Omega)]_{[\theta]},
\]
where $s=(1-\theta) s_0 + \theta s_1, \frac{1}{p} = \frac{1-\theta}{p_0} + \frac{\theta}{p_1}$ and $s_0\neq s_1, s_0, s_1\in \mathbb{R}, 1<p_0, p_1 <\infty$ 
with $0<\theta<1.$
\end{lemma}
\begin{proof}
 See [\cite{Bergh}, Theorem 6.4.5].
\end{proof}

As in \cite{MiT}, for $\Omega\subset\mathbb{R}^3$ an open set and if $1<p<\infty$ and $s\in\mathbb{R},$ we introduce the space 
\[
 H^{s,p}(\curl;\Omega) := \{u\in L_{s}^{p}(\Omega); \ \curl u\in L_{s}^{p}(\Omega)\},
\]
equipped with the natural graph norm
\begin{equation}\label{GRaph}
 \|u\|_{H^{s,p}(\curl;\Omega)} = \|u\|_{L_{s}^{p}(\Omega)} + \|\curl u\|_{L_{s}^{p}(\Omega)}.
\end{equation}
An equivalent norm to \eqref{GRaph} is given by 
\[
 \|u\|_{H^{s,p}(\curl;\Omega)} = \left(\|u\|_{L_{s}^{p}(\Omega)}^{p} + \|\curl u\|_{L_{s}^{p}(\Omega)}^{p}\right)^{{1/p}}.
\]
Under the second norm $H^{s,p}(\curl;\Omega)$ is a Banach space. 
To define the tangential trace along the boundary we need to discuss about the Besov spaces on the boundary. Following \cite{MiT}, we denote
by $L_{1}^{p}(\partial\Omega)$ the Sobolev space of functions in $L^p(\partial\Omega)$ with tangential gradients $\nabla_{\tan}$ $(\nabla_{\tan} := -\nu\wedge(\nu\wedge\nabla))$ in $L^p(\partial\Omega),$
for $1<p<\infty.$ Spaces with fractional smoothness can then be defined via complex interpolation, i.e.,
\[
 L_{\theta}^{p}(\partial\Omega) := [L^p(\partial\Omega),L_{1}^{p}(\partial\Omega)]_{[\theta]}, \ \ \forall \ 0<\theta<1, \ 1<p<\infty.
\]
We also set $$L_{-s}^{p}(\partial\Omega) := (L_{s}^{q}(\partial\Omega))^{'},$$ for all $0\leq s\leq 1, 1<p,q<\infty, {1/p}+{1/q}=1.$\\
On $\partial\Omega,$ the Besov spaces can then be introduced via real interpolation\footnote{More details about this space can be found in \cite{Bergh} and [\cite{Adams}, Chapter 7]. }, i.e.,
\[
 B_{s}^{p,q}(\partial\Omega) := (L^p(\partial\Omega),L_{1}^{p}(\partial\Omega))_{s,q}, \ \text{with}\ 0<s<1,1<p,q<\infty.
\]
Also, for $-1<s<0$ and $1<p,q<\infty$, we set 
\[
 B_{s}^{p,q}(\partial\Omega) := (B_{-s}^{p',q'}(\partial\Omega))^{'}, \ \ {1/p} + {1/q}=1, \ 1/p' + 1/q'=1. 
\]
Now, if $u\in H^{s,p}(\curl;\Omega)$ for some $p, s$ satisfying $1<p<\infty$ and $-1+{1/p}<s<{1/p}$ then we can define the tangential trace $\nu\wedge u\in B_{s-{1/p}}^{p,p}(\partial\Omega)$ by
\[
 \langle \nu\wedge u, \text{Tr} \varphi\rangle := \int_{\Omega}[\langle \curl u,\varphi\rangle - \langle u,\curl\varphi\rangle]dx,
\]
for any $\varphi \in L_{1-s}^{q}(\Omega), {1/p}+{1/q}=1.$ \\
Therefore, the space $H_{0}^{s,p}(\curl;\Omega)$ can be interpreted as
\[
 H_{0}^{s,p}(\curl;\Omega) :=\{u\in H^{s,p}(\curl;\Omega);\ \nu\wedge u=0 \ \text{on}\ \partial\Omega \}.
\]

\begin{lemma}\label{refle}
The space $H_{0}^{s,p}(\curl;\Omega)$ is reflexive, for all $1<p<\infty$ and $s\in (-1+1/p, 1/p)$.
\end{lemma}
\begin{proof}
Define  
\[
\mathcal{F} : H_{0}^{s,p}(\curl;\Omega) \rightarrow L_{s}^{p}(\Omega) \times L_{s}^{p}(\Omega)
\]
by 
\[
\mathcal{F} u :=(u, \curl u). 
\]
The operator $\mathcal{F}$ is bounded linear and isometric. We set $\mathcal{W}$ to be the range of the
operator $\mathcal{F}$, which is a closed subspace of $L_{s}^{p}(\Omega) \times L_{s}^{p}(\Omega)$.
Notice that, the operator 
\[
\mathcal{F} : H_{0}^{s,p}(\curl;\Omega) \rightarrow \mathcal{W}
\]
and its inverse are isometrically isomorphism. 
Since the closed subspace of a reflexive space is reflexive and the isometric isomorphism preserves reflexivity between the spaces, then the proof of the lemma will follow if we show that $L_{s}^{p}(\Omega) \times L_{s}^{p}(\Omega)$ is reflexive.

For $h \in L_{s}^{p}(\Omega) \times L_{s}^{p}(\Omega)$, we set
\[
J_h(F^*) := {}_{(L_{s}^{p}(\Omega) \times L_{s}^{p}(\Omega))^{'}}\langle F^*, h\rangle_{L_{s}^{p}(\Omega) \times L_{s}^{p}(\Omega)}.
\]
%Since, $L_{s}^{p}(\Omega) \times L_{s}^{p}(\Omega)$ is a normed linear space, then by applying Hahn-Banach
%extension theorem, we have
%\[
%\begin{split}
%\|h\|_{L_{s}^{p}(\Omega) \times L_{s}^{p}(\Omega)} 
%& = \sup_{\substack{f\in (L_{s}^{p}(\Omega) \times L_{s}^{p}(\Omega))^{'}\\
% \|F^{*}\|\leq 1 }}
%|{}_{(L_{s}^{p}(\Omega) \times L_{s}^{p}(\Omega))^{'}}\langle F^*, h\rangle_{L_{s}^{p}(\Omega) \times L_{s}^{p}(\Omega)}| \\
%& = \max_{\substack{ f\in (L_{s}^{p}(\Omega) \times L_{s}^{p}(\Omega))^{'} \\
% \|F^{*}\|\leq 1}} |{}_{(L_{s}^{p}(\Omega) \times L_{s}^{p}(\Omega))^{'}}\langle F^*, h\rangle_{L_{s}^{p}(\Omega) \times L_{s}^{p}(\Omega)}|.
%\end{split}
%\]
%Therefore, $J_{h} \in ((L_{s}^{p}(\Omega) \times L_{s}^{p}(\Omega))')' = (L_{s}^{p}(\Omega) \times L_{s}^{p}(\Omega))^{''}$. In fact, 
%\[
%\|J_h\|_{(L_{s}^{p}(\Omega) \times L_{s}^{p}(\Omega))^{''}} = \|h\|_{L_{s}^{p}(\Omega) \times L_{s}^{p}(\Omega)}.
%\]
Let us define the usual canonical mapping 
\[
J : L_{s}^{p}(\Omega) \times L_{s}^{p}(\Omega) \rightarrow (L_{s}^{p}(\Omega) \times L_{s}^{p}(\Omega))^{''}
\]
by $J(h) = J_h$. Notice that $J$ is an isometry. Hence, to show that the Banach space $L_{s}^{p}(\Omega) \times L_{s}^{p}(\Omega)$ is reflexive, we need to prove that the canonical embedding $J$ is surjective.

Given $g \in L_{-s}^{q}(\Omega) \times L_{-s}^{q}(\Omega)$, we define
\[
\tau_q : L_{-s}^{q}(\Omega) \times L_{-s}^{q}(\Omega) \rightarrow (L_{s}^{p}(\Omega) \times L_{s}^{p}(\Omega))'
\]
by
\[
 (\tau_q g)(f) := {}_{L_{-s}^{q}(\Omega) \times L_{-s}^{q}(\Omega)}\langle g, f\rangle_{L_{s}^{p}(\Omega) \times L_{s}^{p}(\Omega)}
\]
for all $f$ in $L_{s}^{p}(\Omega) \times L_{s}^{p}(\Omega)$, where $p, q, s$ satisfy $1<p, q <\infty$,
$1/p + 1/q = 1$ and $-1+1/p <s <1/p$. Then by Lemma \ref{Represation theo}, the operator $\tau_q$ is an isomorphism.

Let us take $F^{**} \in (L_{s}^{p}(\Omega) \times L_{s}^{p}(\Omega))^{''}$. Also we take $\tau_q g$ as $F^{*}$. Then $F^{**}F^{*} = F^{**}\tau_q g$, i.e., 
\[
F^{**}\tau_q : L_{-s}^{q}(\Omega) \times L_{-s}^{q}(\Omega) \rightarrow \mathbb{R}
\]
is a continuous linear map. Hence, Lemma \ref{Represation theo} implies that there exists a unique $h \in L_{s}^{p}(\Omega) \times L_{s}^{p}(\Omega)$ such that
\[
(F^{**}\tau_q)(g) = {}_{L_{s}^{p}(\Omega) \times L_{s}^{p}(\Omega)}\langle h, g\rangle_{L_{-s}^{q}(\Omega) \times L_{-s}^{q}(\Omega)}.
\]
Therefore, 
\[
\begin{split}
F^{**}F^{*} 
&= (F^{*}\tau_q)(g) \\
& = {}_{L_{s}^{p}(\Omega) \times L_{s}^{p}(\Omega)}\langle h, g\rangle_{L_{-s}^{q}(\Omega) \times L_{-s}^{q}(\Omega)} \\
& =  {}_{L_{-s}^{q}(\Omega) \times L_{-s}^{q}(\Omega)}\langle g, h\rangle_{L_{s}^{p}(\Omega) \times L_{s}^{p}(\Omega)} \\
& = (\tau_q g)(h) \\
& =  {}_{(L_{s}^{p}(\Omega) \times L_{s}^{p}(\Omega))^{'}}\langle F^*, h\rangle_{L_{s}^{p}(\Omega) \times L_{s}^{p}(\Omega)} \\
& = J_h(F^{*})
\end{split}
\]
i.e., $F^{**} = J_h = J(h)$. Hence $J$ is surjective.
\end{proof}
We finish this section with the following lemma where we state a Kato-Ponce type inequality.
\begin{lemma}\label{KP}
Assume that $f\in L_{s}^{p}(\Omega)$ and $g\in W^{s,\infty}(\Omega),$ then $fg\in L_{s}^{p}(\Omega)$ with the following estimate
\[
\|fg\|_{L_{s}^{p}(\Omega)} \leq C \|f\|_{L_{s}^{p}(\Omega)} \|g\|_{W^{s,\infty}(\Omega)},
\]
where $C=C(s,p)>0$ for all $1<p<\infty$ and $s\geq 0.$
\end{lemma}
\begin{proof}
For $s=0,$ the proof is trivial and $C=1$. So, we consider the case $s>0$.
Let us first extend the functions $f$ and $g$ by zero outside $\Omega$. Following \cite{stein}, we define the Bessel potentials 
$\mathscr{J}_s$ by $\mathscr{J}_s = (I-\Delta)^{-s/2}$ for $s>0$ and recall that, 
\[
L_{s}^{p}(\mathbb{R}^3) = \mathscr{J}_s(L^p(\mathbb{R}^3)), \ \ 1\leq p\leq\infty, \ \ s\geq 0.
\]
Hence,
\[
\begin{split}
\|fg\|_{L_{s}^{p}(\Omega)} 
& = \|fg\|_{L_{s}^{p}(\mathbb{R}^3)} \\
& = \|(I-\Delta)^{s/2}(fg)\|_{L^p(\mathbb{R}^3)}. 
\end{split}
\]
Now we recall the Kato-Ponce inequality, see for instance in \cite{kato}, as 
\[
\|J^s(fg)\|_{L^p(\mathbb{R}^3)} \leq C [\|f\|_{L^p(\mathbb{R}^3)}\|J^sg\|_{L^{\infty}(\mathbb{R}^3)}
+\|J^sf\|_{L^p(\mathbb{R}^3)}\|g\|_{L^{\infty}(\mathbb{R}^3)}],
\]
where $s>0$, $1<p<\infty$ and $J^s := (I-\Delta)^{s/2}$ with the constant $C=C(s,p)>0.$
So, we obtain
\begin{equation}\label{eq1}
\|fg\|_{L_{s}^{p}(\Omega)} \leq C [\|f\|_{L^p(\mathbb{R}^3)}\|J^sg\|_{L^{\infty}(\mathbb{R}^3)}
+\|J^sf\|_{L^p(\mathbb{R}^3)}\|g\|_{L^{\infty}(\mathbb{R}^3)}].
\end{equation}
Since, $J^s$ is an isomorphism between $L_{s}^{p}(\mathbb{R}^3)$ and $L^p(\mathbb{R}^3)$ for $1\leq p\leq \infty$ and $s\in\mathbb{R}$, see [\cite{Bergh}, Theorem 6.2.7], then we have
\begin{equation}\label{eq2}
\|J^sg\|_{L^{\infty}(\mathbb{R}^3)} \leq C \|g\|_{W^{s,\infty}(\mathbb{R}^3)},
\end{equation}
where $C=C(s)>0.$
Also note that, for $s>0$, $L_{s}^{p}(\mathbb{R}^3)$ is a subspace of $L^p(\mathbb{R}^3)$, i.e., for any $f\in L_{s}^{p}(\mathbb{R}^3)$ we have
\begin{equation}\label{eq3}
\|f\|_{L^p(\mathbb{R}^3)} \leq C \|f\|_{L_{s}^{p}(\mathbb{R}^3)}.
\end{equation}
Combining \eqref{eq1}, \eqref{eq2} and \eqref{eq3}, we obtain
\[
\|fg\|_{L_{s}^{p}(\Omega)} \leq C\|f\|_{L_{s}^{p}(\mathbb{R}^3)} \|g\|_{W^{s,\infty}(\mathbb{R}^3)}
\]
i.e.,
\[
\|fg\|_{L_{s}^{p}(\Omega)} \leq C\|f\|_{L_{s}^{p}(\Omega)} \|g\|_{W^{s,\infty}(\Omega)},
\]
where $C=C(s,p)>0$.
\end{proof}

\end{section}

\section{Main result}\label{theosec}
We start the section by defining a region $R_{\Omega}$ as follows:
\[
 (s,{1/p})\in R_{\Omega} \Leftrightarrow 
%\left \{ 
\begin{cases}
0<\frac{1}{p}<1, \\
-1+\frac{1}{p}<s<\frac{1}{p},\\ 
\frac{2}{3}(1-\frac{1}{p_{\Omega}})< \frac{1}{p} - \frac{s}{3}<\frac{1}{3}(\frac{2}{p_{\Omega}}+1).
\end{cases}
%\right \}
\]
Remark that, $R_{\Omega}$ can be determined by the geometric character of the domain $\Omega.$
Here, $p_{\Omega}$ is the H{\"o}lder conjugate exponent of $q_{\Omega}$ and $q_{\Omega}$ is the supremum of all $q$ so that the Dirichlet 
and Neumann problem for the Laplace-Beltrami operator in $\Omega$ is well-posed in $W^{1,q}$ spaces. However, $1\leq p_{\Omega}<2$ when $\partial\Omega$
is Lipschitz regular and $p_{\Omega} = 1$ when $\partial\Omega \in C^1$. A more detailed explanation can be found in \cite{MiT}. In the next 
sections we use the notations $R_{\Omega}^{+}$ and $R_{\Omega}^{-}$ given by 
\[
R_{\Omega}^{+} := R_{\Omega}\cap\{(s,1/p); s\geq 0, 1<p<\infty\}
\]
and 
\[
R_{\Omega}^{-} := R_{\Omega}\cap\{(s,1/p); s<0, 1<p<\infty\}.
\]
Now, we are in a position to state the following theorem as our main result.
\begin{theorem}\label{pert}
Let $\Omega$ be a bounded and Lipschitz domain in $\mathbb{R}^3$. Let the coefficient $a$ be a $3 \times 3$ symmetric  matrix, with elements in $W^{s, \infty}(\Omega)$, satisfying the uniform ellipticity condition, i.e.,
there exist positive constants $m,  M$ such that
\begin{equation}\label{ellipticity-theorem}
m |\xi|^2 \leq a(x)\xi\cdot\overline{\xi} \leq M |\xi|^2,
\end{equation}
for all $\xi\in\mathbb{C}^3$ and almost every $x\in\Omega$ and having, if $s>0$, a bounded H{\"o}lder semi-norm with exponent $s$, i.e., there exists $\tilde M>0$ such that 
\footnote{In the case $s=0$, this condition is not needed. In all subsequent estimates, we can replace $\tilde{M}$ by $0$ in this case.}
\begin{equation}\label{Holder-regularity-bound}
|a|_{C^{0,s}}\leq \tilde M.
\end{equation}
% We also assume that $0$ is not an eigenvalue of the problem
% \begin{equation}
% \begin{cases}
% \mbox{Find } u \in H_0(\curl; \Omega):=H_0^{0, 2}(\curl; \Omega)\; \mbox{such that}\\ 
% \curl (a(x)\curl u) + k^2u = 0, \ \text{in}\ \Omega \\
% \nu\wedge u = 0, \ \text{on}\ \partial\Omega. 
% \end{cases}
% \end{equation}

Then for any $f\in (H_{0}^{-s,q}(\curl;\Omega))'$, the following problem
\begin{equation}\label{max3}
\begin{cases}
 \curl (a(x)\curl u) + k^2 u = f,\ \text{in}\ \Omega\\
 \nu\wedge u = 0,\ \text{on}\ \partial\Omega
\end{cases}
\end{equation}
has one and unique solution in $H_{0}^{s,p}(\curl;\Omega)$ for all $(s,{1/p})\in S \ (:= S^{+}\cup S^{-})$
where 
\begin{equation}\label{S}
S^{+} :=\bigcup_{(s_0,1/{p_0})\in R_{\Omega}^{+} }
\left \{
  \left(s,\frac{1}{p} \right)\in R_{\Omega}^{+}  ;
\begin{array}{l}
s=(1-\theta)s_0,\\
\frac{1}{p}=\frac{1-\theta}{p_0} +\frac{\theta}{2}, \mbox{ where } \theta \in (0, 1) \mbox{ is such that }\\ (1-\theta)\log \mathcal M_{s_0,p_0} +\log k_0(s,p) <0. 
\end{array}
\right \}
\end{equation}
with $$\mathcal{M}_{s_0,p_0} := \|\mathcal{K}_{s_0,p_0}^{-1}\|_{(H_{0}^{-s_0,q_0}(\curl;\Omega))' \rightarrow H_{0}^{s_0,p_0}(\curl;\Omega)}$$ 
under the condition that 
\begin{equation}\label{conditions-for-estimates}
k_0(s,p) := \max\{|1-\frac{mk^2}{M^2}|, C(s,p)[1-\frac{m^2}{M^2} + \frac{m\tilde M}{M^2}]\}<1.
\end{equation}
Here $\mathcal{K}_{s_0,p_0}u = \curl\curl u + u$ and $C(s,p)$ is the constant appearing in the Kato-Ponce inequality in Lemma \ref{KP}. 
The region $S^{-}$ is given by
\begin{equation}\label{negS}
S^{-} := \{(s,1/p)\in R_{\Omega}^{-} ; (-s,1/q)\in S^{+} \}.
\end{equation}

In addition, the solution satisfies the following estimate 
\[
 \|u\|_{L_{s}^{p}(\Omega)} + \|\curl u\|_{L_{s}^{p}(\Omega)} \leq C \|f\|_{(H_{0}^{-s,q}(\curl;\Omega))'}.
\]
\end{theorem}

\begin{figure}
\centering
\includegraphics[width=0.8\textwidth]{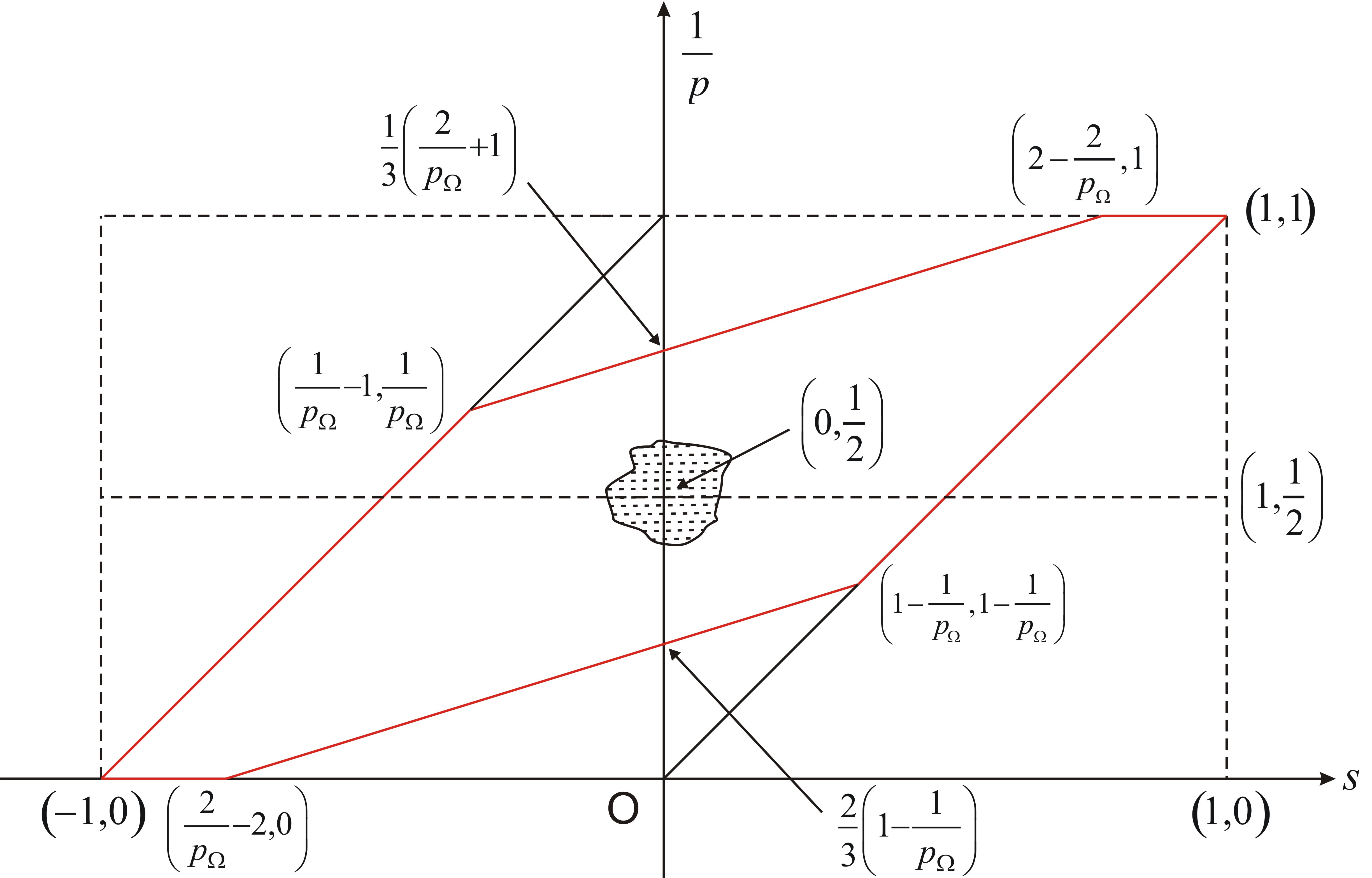}
\caption{The interior of the hexagon bounded by the red lines is the region $R_{\Omega}$ and the interior of 
the dashed region, denoted by $S$ and characterized in (\ref{S}) and \eqref{negS}, represents the well-posedness region for the Maxwell problem \eqref{max3}.}
\label{Fig 1}
\end{figure}

\bigskip

In Figure 1, considered in the $(s,1/p)$-plane (i.e., smoothness vs reciprocal integrability), the dashed area
represents the well-posedness region for the Maxwell problem \eqref{max3}. We first fix $(s_0,1/{p_0}) \in R_{\Omega}^{+}$. 
% Then by interpolation we have that the points $(s,1/p)$ will be on the straight line joining $(0,1/2)$ and $(s_0,1/{p_0})$. 
The property 
\begin{equation}\label{actual-condition}
(1-\theta)\log \mathcal M_{s_0,p_0} +\log k_0(s,p) <0
\end{equation}
says that the points $(s,1/p)$ will be laying on some part of this straight line joining $(0,1/2)$ and $(s_0,1/{p_0})$. 
Now, if we take any other point $(s_0,1/{p_0})\in R_{\Omega}^{+}$ and use the same argument, then we end up with the positive $s$-part 
of dashed region $S$ i.e., $S^{+}$, in Figure \ref{Fig 1}, where $\mathcal{M}_{s,p}k_{0}^{\frac{1}{p}} (s,p)<1$. 
The well-posedness region $S^{-}$ can be obtained by using duality argument on the Maxwell operator and the symmetry of the matrix $a$.

\bigskip

Theorem \ref{pert} is proved in two steps. In the first step, we deal with the unperturbed problem and in the second step, we discuss the perturbed problem.
In the unperturbed case, we consider the coefficient $a$ to be the identity matrix $I$. Then the system \eqref{max3} reduces to the well known time harmonic Maxwell
model with constant permittivity and permeability and the regularity for the solutions of this type of model has been derived in \cite{MiT}.  
In the perturbed case, we follow the approach by Gr{\"o}ger, see \cite{Groe}, based on the Banach fixed point theorem. 

% Remark that, in our work we need this region $R_{\Omega}$ in order to obtain a sharp estimate for the unperturbed problem. However, the restrictions for $s,p$ to the regions $S_1, S_2, S_3$ and
% $S_4$ are only because of the way we are dealing with the perturbed case. 

\begin{remark}\label{general-case}
The result in Theorem \ref{pert} could be extended to obtain the well-posedness in $H_0^{s,p}(\curl,\Omega)$ for the following problem
\begin{equation}\label{max7}
\begin{cases}
 \curl (a(x)\curl u) + b(x) u = f,\ \text{in}\ \Omega\\
 \nu\wedge u = 0,\ \text{on}\ \partial\Omega
\end{cases}
\end{equation}
where $a \in (W^{s, \infty}(\Omega))^{3\times 3}$ satisfies the conditions (\ref{ellipticity-theorem}) and (\ref{Holder-regularity-bound}) and $ b \in (L^{\infty}(\Omega))^3$. 
One way to prove this is to show that the solution operator of this problem is a compact perturbation of the solution operator of the problem (\ref{max3})
and then use the Fredholm alternative, as it is done in \cite{Monk} for the case $s=0$ and $p=2$. 
\end{remark}
\begin{remark}\label{link-s-p-k}
We make the following observations.
%  Theorem \ref{pert} says that under the conditions (\ref{conditions-for-estimates}) and (\ref{actual-condition}) on the apriori bounds $m, M, \tilde{M}, k$ 
%  and the indices $s$ and $p$, where $s=(1- \theta)s_0$ and 
%  $\frac{1}{p}:=\frac{1-\theta}{p_0}+\frac{\theta}{2}$ with $\theta \in [0,\; 1]$ and $(s_0, \frac{1}{p_0}) \in R^+_\Omega$ \footnote{This is the case when $s\geq 0$. 
%  The case $s\leq 0$ can be dealt with by duality as explained above.}, we have the well posedness
%  of the problem (\ref{max3}) in $H_{0}^{s,p}(\curl;\Omega)$ for $f \in (H_{0}^{-s,q}(\curl;\Omega))'$ where $p^{-1}+q^{-1}=1$. We make the following observations.
 \begin{enumerate}
 
\item The regularity result Theorem \ref{pert} can be understood as follows. Let be given the bounds $m, M$ and 
 $\tilde{M}$. Assume, in addition, that $k^2$ is such that $\vert 1-\frac{m\; k^2}{M^2}\vert <1$. Then (\ref{max3}) is well posed in $H_{0}^{s,p}(\curl;\Omega)$ for $s$ and $p$,
 of the form $s=(1- \theta)s_0$ and 
 $\frac{1}{p}:=\frac{1-\theta}{p_0}+\frac{\theta}{2}$ with $(s_0, \frac{1}{p_0}) \in R^+_\Omega$, such that 
 \begin{equation}\label{sp--mM-1}
  C(s, p)( 1-\frac{ m^2}{M^2}+\frac{m \tilde{M}}{M^2})<1
 \end{equation}
 and 
 \begin{equation}\label{sp--mM-2}
  (1-\theta)\log \mathcal M_{s_0,p_0} +\log k_0(s,p) <0
 \end{equation}
 where 
 \begin{equation}\label{k-remainded}
  k_0(s,p) = \max\{|1-\frac{mk^2}{M^2}|, C(s,p)[1-\frac{m^2}{M^2} + \frac{m\tilde M}{M^2}]\}.
 \end{equation}
%  The condition (\ref{sp--mM-1}) makes sense and it is satisfied for instance for $s=0$ and $p\geq 1$ since in this case $C(s, p)=1$ and we have assumed that 
%  $\vert 1-\frac{ m^2}{M^2}\vert+\frac{m \tilde{M}}{M^2}<1$.
 
The extra condition on $k^2$, i.e., $\vert 1-\frac{m\; k^2}{M^2}\vert <1$, can be removed by combining this result and Remark \ref{general-case}. 

\item In the case $s=0$, we have $C(s, p)=1$. In addition, in this case, we can take $\tilde{M}=0$, i.e., we assume the elements of $a$ to be in $L^{\infty}(\Omega)$ only. Then the condition 
(\ref{sp--mM-1}) reduces to  $\vert 1-\frac{ m^2}{M^2}\vert <1$, which is trivially satisfied, and then the condition 
(\ref{sp--mM-2}) characterizes the range of $p$ for which we have well posedness.

 \item In the case where $a$ is a constant coefficient equal to the identity matrix we can take $\frac{m}{M}\rightarrow 1$ and $\tilde{M} \rightarrow 0$ and then  
 $C(s, p)( 1-\frac{ m^2}{M^2}+\frac{m \tilde{M}}{M^2})<<1$. This means that in the case of identity coefficient $a$, $S^+$ and $S^-$ become $R^+_\Omega$ and 
 $R^-_\Omega$ respectively and hence $S=R_\Omega$. This reduces to the result in \cite{MiT}.
 \end{enumerate}
\end{remark}

\section{Proof of the $H^{s,p}(\curl;\Omega)$ estimates for the Maxwell system}\label{proofsec}
We begin this section with the following lemma to characterize the dual space of $H_{0}^{-s,q}(\curl;\Omega)$ with an appropriate range of $s$ and $q$.
\begin{lemma}\label{repre}
 Assume that $\varphi\in (H_{0}^{-s,q}(\curl;\Omega))^{'},$ then $\varphi$ can be uniquely written as $\varphi = g_1 + \curl g_2,$
with the estimate
\[
 \|g_1\|_{L_{s}^{p}(\Omega)} + \|g_2\|_{L_{s}^{p}(\Omega)} \leq C \|\varphi\|_{(H_{0}^{-s,q}(\curl;\Omega))^{'}},
\]
where $g_1,g_2 \in L_{s}^{p}(\Omega),$ $1/p+1/q = 1,$ $1<p<\infty$ and $-1+1/p<s<1/p.$
\end{lemma}

\begin{proof}
 The operator $$P : H_{0}^{-s,q}(\curl;\Omega) \rightarrow L_{-s}^{q}(\Omega)\times L_{-s}^{q}(\Omega),$$ defined by
$$Pu := (u, \curl u),$$ is linear, bounded and isometric. 
Also we define $W := P(H_{0}^{-s,q}(\curl;\Omega))$, which is a closed subspace of $L_{-s}^{q}(\Omega)\times L_{-s}^{q}(\Omega).$

Note that the adjoint operator $$P^{*} : W^{'} \rightarrow (H_{0}^{-s,q}(\curl;\Omega))^{'}$$ is invertible and continuous.
Hence for given $\varphi\in (H_{0}^{-s,q}(\curl;\Omega))^{'}$, there exists a unique $\varphi^*\in W'$ such that
$P^{*}\varphi^* = \varphi$.
Now, by Hahn-Banach extension theorem, there exists a linear functional (name it $\tilde\varphi^{*}$), $$\tilde\varphi^{*} : L_{-s}^{q}(\Omega)\times L_{-s}^{q}(\Omega) \rightarrow \mathbb{R}$$
such that 
\begin{equation}\label{Hahn}
_{(L_{-s}^{q}(\Omega)\times L_{-s}^{q}(\Omega))'}\langle \tilde\varphi^{*}, u\rangle_{L_{-s}^{q}(\Omega)\times L_{-s}^{q}(\Omega)}
= {}_{W'}\langle \varphi^{*}, u \rangle_W, \  \forall \ u\in W
\end{equation}
and 
$$\|\tilde\varphi^{*}\|_{(L_{-s}^{q}(\Omega)\times L_{-s}^{q}(\Omega))'} = \|\varphi^{*}\|_{W'}.$$

\bigskip

On the other hand, since $\tilde\varphi^{*} \in (L_{-s}^{q}(\Omega)\times L_{-s}^{q}(\Omega))'$, then by Lemma \ref{Represation theo}, for all $-1+1/p<s<1/p$,
there exists a unique $g :=(g_1, g_2) \in L_{s}^{p}(\Omega)\times L_{s}^{p}(\Omega)$ such that,
 \[
 \tilde\varphi^{*} f 
= {}_{L_{s}^{p}(\mathbb{R}^3)}\langle \tilde{g_1}, \tilde{f_1}\rangle_{L_{-s}^{q}(\mathbb{R}^3)} + {}_{L_{s}^{p}(\mathbb{R}^3)}\langle \tilde{g_2}, \tilde{f_2}\rangle_{L_{-s}^{q}(\mathbb{R}^3)},
 \]
for all $\tilde{f_1}, \tilde{f_2} \in L_{-s}^{q}(\mathbb{R}^3)$, 
 with the norm estimate
 \begin{equation}\label{Riesz}
 c_1 \|g\|_{L_{s}^{p}(\Omega) \times L_{s}^{p}(\Omega)}\leq \|\tilde\varphi^{*}\|_{(L_{-s}^{q}(\Omega)\times L_{-s}^{q}(\Omega))'} \leq c_2 \|g\|_{L_{s}^{p}(\Omega) \times L_{s}^{p}(\Omega)},
 \end{equation}
where $\tilde{f_j}$ is any extension of $f_j$ satisfying $\mathcal{R}_{\Omega}\tilde{f_j} = f_j,$ and $\tilde g$ be an extension of $g$ by zero outside $\Omega$ for all $j=1,2.$
Combining \eqref{Hahn} and \eqref{Riesz} we have 
$$\|\varphi^{*}\|_{W'} = \|\tilde\varphi^{*}\|_{(L_{-s}^{q}(\Omega)\times L_{-s}^{q}(\Omega))'} \geq c_1 \|g\|_{L_{s}^{p}(\Omega) \times L_{s}^{p}(\Omega)}.$$
Therefore, for $v\in H_{0}^{-s,q}(\curl;\Omega)$, we have
\begin{equation}\label{R1pre}
 \begin{split}
  {}_{(H_{0}^{-s,q}(\curl;\Omega))'}\langle \varphi, v\rangle_{H_{0}^{-s,q}(\curl;\Omega)} 
& = {}_{W'}\langle \varphi^{*}, Pv\rangle_W \\
& = {}_{(L_{-s}^{q}(\Omega)\times L_{-s}^{q}(\Omega))'}\langle \tilde\varphi^{*}, Pv\rangle_{L_{-s}^{q}(\Omega)\times L_{-s}^{q}(\Omega)}\\
& = {}_{L_{s}^{p}(\mathbb{R}^3)}\langle\tilde{g_1},\tilde{v}\rangle_{L_{-s}^{q}(\mathbb{R}^3)} + {}_{L_{s}^{p}(\mathbb{R}^3)}\langle \tilde{g_2},\widetilde{\curl v}\rangle_{L_{-s}^{q}(\mathbb{R}^3)},\\
\end{split}
\end{equation}
where, tilde denotes the extension by zero outside $\Omega$. 

Now we show that
\begin{equation}\label{integration by parts}
_{L_{s}^{p}(\Omega)}\langle g_2,\curl v \rangle_{L_{-s}^{q}(\Omega)} = {}_{(H_{0}^{-s,q}(\curl;\Omega))'}\langle \curl g_2, v \rangle_{H_{0}^{-s,q}(\curl;\Omega)}.
\end{equation}
Recall that, for $-1+1/p < s<1/p,$ $L_{s,0}^{p}(\Omega) = L_{s}^{p}(\Omega)$. As $g_2\in L_{s,0}^{p}(\Omega),$
then, $g_2$ is a distribution such that $g_2\in L_{s}^{p}(\mathbb{R}^3)$ with support in $\overline{\Omega}$. Using the distribution derivative
of $g_2$, we define $\curl g_2$ by
\[
_{(C_{0}^{\infty}(\Omega))'}\langle\curl g_2, v \rangle_{C_{0}^{\infty}(\Omega)} := {}_{(C_{0}^{\infty}(\Omega))'}\langle g_2,\curl v \rangle_{C_{0}^{\infty}(\Omega)} \ \forall \ v\in C_{0}^{\infty}(\Omega).
\]
It is clear that, $\curl g_2\in (C_{0}^{\infty}(\Omega))'$. Also recall that, $C_{0}^{\infty}(\Omega)$
is dense in $H_{0}^{-s,q}(\curl;\Omega),$ see [\cite{MiT}, A.27].
Therefore, we can define $\curl g_2$ on $H_{0}^{-s,q}(\curl;\Omega)$ by
\[
\langle\curl g_2,v\rangle := \lim_{m\rightarrow \infty}\langle\curl g_2,v_m\rangle, 
\]
where $v_m \rightarrow v$ in $H_{0}^{-s,q}(\curl;\Omega)$ with $v_m\in C_{0}^{\infty}(\Omega)$.
Hence $\curl g_2$ defines a bounded linear functional on $H_{0}^{-s,q}(\curl;\Omega)$ and we obtain
\begin{equation}\label{iingra}
\begin{split}
{}_{(H_{0}^{-s,q}(\curl;\Omega))'}\langle\curl g_2,v\rangle_{H_{0}^{-s,q}(\curl;\Omega)} 
& = \lim_{m\rightarrow \infty}{}_{(C_{0}^{\infty}(\Omega))'}\langle\curl g_2,v_m\rangle_{C_{0}^{\infty}(\Omega)} \\
& = \lim_{m\rightarrow \infty}{}_{(C_{0}^{\infty}(\Omega))'}\langle g_2,\curl v_m\rangle_{C_{0}^{\infty}(\Omega)} \\
& = {}_{L_{s}^{p}(\Omega)}\langle g_2,\curl v\rangle_{L_{-s}^{q}(\Omega)}. 
\end{split}
\end{equation}
Therefore, combining the equations \eqref{R1pre} and \eqref{iingra}, we obtain
\[
{}_{(H_{0}^{-s,q}(\curl;\Omega))'}\langle \varphi, v\rangle_{H_{0}^{-s,q}(\curl;\Omega)}
={}_{(H_{0}^{-s,q}(\curl;\Omega))'}\langle g_1 + \curl g_2, v\rangle_{H_{0}^{-s,q}(\curl;\Omega)}.
\]
% & = {}_{L_{s}^{p}(\Omega)\times L_{s}^{p}(\Omega)}\langle g, Pv\rangle_{L_{-s}^{q}(\Omega)\times L_{-s}^{q}(\Omega)} \\
% & = {}_{L_{-s}^{p}(\Omega)\times L_{-s}^{p}(\Omega)}\langle (g_1,g_2), (v, \curl v) \rangle_{L_{s}^{q}(\Omega)\times L_{s}^{q}(\Omega)} \\
% &= {}_{L_{-s}^{p}(\Omega)}\langle g_1,v \rangle_{L_{s}^{q}(\Omega)} + {}_{L_{-s}^{p}(\Omega)}\langle g_2,\curl v\rangle_{L_{s}^{q}(\Omega)},\\
% &\text{integrating by parts and using $\nu\wedge v = 0$ on $\partial \Omega$, we obtain}\\
% & = {}_{(H_{0}^{s,q}(\curl;\Omega))'}\langle g_1,v\rangle_{H_{0}^{s,q}(\Omega)} + {}_{(H_{0}^{s,q}(\curl;\Omega))'}\langle\curl g_2,v\rangle_{H_{0}^{s,q}(\Omega)} \\
% & = \langle g_1 + \curl g_2, v\rangle_{((H_{0}^{s,q}(\curl;\Omega))',H_{0}^{s,q}(\curl;\Omega))}.
%  \end{split}
% \]
Hence, for any $\varphi \in (H_{0}^{-s,q}(\curl;\Omega))'$, there exist unique $g_1\in L_{s}^{p}(\Omega)$ and $g_2\in L_{s}^{p}(\Omega)$ such that $\varphi$
has a representation $\varphi = g_1 + \curl g_2.$\\
In addition, we have the following estimate.
\[
 \begin{split}
  \|\varphi\|_{(H_{0}^{-s,q}(\curl;\Omega))'}
& = \sup_{\|u\|_{H_{0}^{-s,q}(\curl;\Omega)}\leq 1}|\langle \varphi, u\rangle| \\
& = \sup_{\|u\|_{H_{0}^{-s,q}(\curl;\Omega)}\leq 1} |{}_{W'}\langle \varphi^{*}, Pu\rangle_W| \\
& = \sup_{\|w\|_W\leq 1} |{}_{W'}\langle \varphi^{*}, w\rangle_W| \\
& = \|\varphi^{*}\|_{W'} = \|\tilde\varphi^{*}\|_{(L_{-s}^{q}(\Omega)\times L_{-s}^{q}(\Omega))'} \geq c_1 \|g\|_{L_{s}^{p}(\Omega)\times L_{s}^{p}(\Omega)}\\
& = c_1\left[\|g_1\|_{L_{s}^{p}(\Omega)}^{p} + \|g_2\|_{L_{s}^{p}(\Omega)}^{p}\right]^{1/p}\\
& \geq C \{\|g_1\|_{L_{s}^{p}(\Omega)} + \|g_2\|_{L_{s}^{p}(\Omega)}\},
 \end{split}
\]
i.e.,
\begin{equation}\label{estR}
 \|g_1\|_{L_{s}^{p}(\Omega)} + \|g_2\|_{L_{s}^{p}(\Omega)} \leq C \|\varphi\|_{(H_{0}^{-s,q}(\curl;\Omega))'}.
\end{equation}
\end{proof}

\subsection{Unperturbed problem}
\begin{theorem}\label{imPTH}
 Assume $\Omega$ to be a bounded Lipschitz domain.
 For a given $f\in (H_{0}^{-s,q}(\curl;\Omega))',$ there exists a unique $u\in H_{0}^{s,p}(\curl;\Omega)$  satisfying the following Maxwell problem
\begin{equation}\label{impMax}
 \begin{cases}
  \curl\curl u + u = f, \text{in}\ \Omega\\
\nu\wedge u = 0, \text{on}\ \partial\Omega,
 \end{cases}
\end{equation}
for all $(s,1/p)\in R_{\Omega}$ and ${1/p} + {1/q} = 1.$

In addition, we have the estimate
\[
 \|u\|_{L_{s}^{p}(\Omega)} + \|\curl u\|_{L_{s}^{p}(\Omega)} \leq C \|f\|_{(H_{0}^{-s,q}(\curl;\Omega))'},
\]
for all $(s,1/p)\in R_{\Omega}$ and ${1/p} + {1/q} = 1.$
\end{theorem}
\begin{proof}
 Since $f\in (H_{0}^{-s,q}(\curl;\Omega))'$, then from Lemma \ref{repre} there exist a unique $g_1\in L_{s}^{p}(\Omega)$
and $g_2 \in L_{s}^{p}(\Omega)$ such that $f = g_1 + \curl g_2$ with the estimate
\begin{equation}\label{esImp}
 \|g_1\|_{L_{s}^{p}(\Omega)} + \|g_2\|_{L_{s}^{p}(\Omega)} \leq C \|f\|_{(H_{0}^{-s,q}(\curl;\Omega))'}.
\end{equation}
Therefore the problem \eqref{impMax} can be viewed as 
\begin{equation}\label{impMaxx}
 \begin{cases}
  \curl\curl u + u = g_1 + \curl g_2, \text{in}\ \Omega\\
\nu\wedge u = 0, \text{on}\ \partial\Omega.
 \end{cases}
\end{equation}
Define, $u := g_1 + \curl v$ and $v := g_2 - \curl u.$ Then the \eqref{impMaxx} reduces to the following problem
\begin{equation}\label{impMaxxx}
 \begin{cases}
  \curl u + v = g_2,\ \text{in}\ \Omega \\
  \curl v - u = -g_1,\ \text{in}\ \Omega \\
  \nu\wedge u = 0\ \text{on}\ \partial\Omega. 
 \end{cases}
\end{equation}
In \cite{MiT}, it is shown that the problem \eqref{impMaxxx} is well posed, i.e., there exists a unique $u\in H_{0}^{s,p}(\curl;\Omega)$ satisfying the problem
\eqref{impMaxxx} together with the estimate
\begin{equation}\label{mit}
 \|u\|_{L_{s}^{p}(\Omega)} + \|\curl u\|_{L_{s}^{p}(\Omega)} \leq C \left(\|g_1\|_{L_{s}^{p}(\Omega)} + \|g_2\|_{L_{s}^{p}(\Omega)}\right),
\end{equation}
for all $(s,{1/p})\in R_{\Omega}$.
Finally, combining \eqref{esImp} and \eqref{mit}, we have the required estimate
\[
 \|u\|_{L_{s}^{p}(\Omega)} + \|\curl u\|_{L_{s}^{p}(\Omega)} \leq C \|f\|_{(H_{0}^{-s,q}(\curl;\Omega))'},
\]
for all $(s,{1/p})\in R_{\Omega}$ and ${1/p} + {1/q} = 1.$
\end{proof}
To deal with the case of the perturbed problem, we follow Gr{\"o}ger's approach, see \cite{Groe}.

\subsection{Perturbed problem}

Before proving Theorem \ref{pert}, we state and justify some intermediate lemmas. 
Define, $$\mathcal{L}_{s,p} : H_{0}^{s,p}(\curl;\Omega) \rightarrow L_{s}^{p}(\Omega)\times L_{s}^{p}(\Omega)$$ by
$$ \mathcal{L}_{s,p} u := \left
     (\begin{array}{c} 
               u \\ 
               \curl u
     \end{array}
      \right).$$
Remark that $\mathcal{L}_{s,p}$ is an isometry. Let us characterize its adjoint. Consider the functions $u\in (C_{0}^{\infty}(\mathbb{R}^3))^6$ and $v\in (C_{0}^{\infty}(\mathbb{R}^3))^6$, which are compactly supported in $\Omega.$ Also take $v$ of the form 
$v := \left
(\begin{array}{c} 
               a \\ 
               A
     \end{array}
      \right),$
then
\[
 \begin{split}
 & {}_{L_{s}^{p}(\Omega)\times L_{s}^{p}(\Omega)}\langle \mathcal{L}_{s,p}u, v\rangle_{L_{-s}^{q}(\Omega)\times L_{-s}^{q}(\Omega)} \\
& ={}_{L_{s}^{p}(\Omega)\times L_{s}^{p}(\Omega)}\langle \left
(\begin{array}{c} 
               u \\ 
               \curl u
     \end{array}
      \right),
\left
(\begin{array}{c} 
               a \\ 
               A
     \end{array}
      \right)
\rangle_{L_{-s}^{q}(\Omega)\times L_{-s}^{q}(\Omega)}  \\
& = {}_{L_{s}^{p}(\Omega)}\langle u, a\rangle_{L_{-s}^{q}(\Omega)} 
+ {}_{L_{s}^{p}(\Omega)}\langle \curl u, A\rangle_{L_{-s}^{q}(\Omega)} \\
& = {}_{H_{0}^{s,p}(\curl;\Omega)}\langle u, a\rangle_{(H_{0}^{s,p}(\curl;\Omega))'}
  + {}_{H_{0}^{s,p}(\curl;\Omega)}\langle u, \curl A\rangle_{(H_{0}^{s,p}(\curl;\Omega))'} \\
& = {}_{H_{0}^{s,p}(\curl;\Omega)}\langle u, a + \curl A\rangle_{(H_{0}^{s,p}(\curl;\Omega))'}\\
& = {}_{H_{0}^{s,p}(\curl;\Omega)}\langle u, \mathcal{L}_{s,p}^{*}v\rangle_{(H_{0}^{s,p}(\curl;\Omega))'}.
\end{split}
\]
Since $C_{0}^{\infty}(\Omega)$ is dense in $H_{0}^{s,p}(\curl;\Omega),$ so for any $u\in H_{0}^{s,p}(\curl;\Omega)$
 the equality 
\[
 {}_{L_{s}^{p}(\Omega)\times L_{s}^{p}(\Omega)}\langle \mathcal{L}_{s,p}u, v\rangle_{L_{-s}^{q}(\Omega)\times L_{-s}^{q}(\Omega)}
= {}_{H_{0}^{s,p}(\curl;\Omega)}\langle u, \mathcal{L}_{s,p}^{*}v\rangle_{(H_{0}^{s,p}(\curl;\Omega))'}.
\]
holds for all $v\in L_{-s}^{q}(\Omega)\times L_{-s}^{q}(\Omega).$
Therefore, the adjoint of $\mathcal{L}_{s,p}$ can be characterized as follows
\[
\mathcal{L}_{s,p}^{*} : L_{-s}^{q}(\Omega)\times L_{-s}^{q}(\Omega) \rightarrow (H_{0}^{s,p}(\curl;\Omega))'
\]
with
\[
 \mathcal{L}_{s,p}^{*}\left
(\begin{array}{c} 
               a \\ 
               A
     \end{array}
      \right)
= a + \curl A.
\] 
Similarly, we have $\mathcal{L}_{-s,q}^{*} : L_{s}^{p}(\Omega)\times L_{s}^{p}(\Omega) \rightarrow (H_{0}^{-s,q}(\curl;\Omega))'$
with $\mathcal{L}_{-s,q}^{*}\left
(\begin{array}{c} 
               a \\ 
               A
     \end{array}
      \right)
= a + \curl A.$
Finally, we define 
\[
\mathcal{K}_{s,p} := \mathcal{L}_{-s,q}^{*}\mathcal{L}_{s,p}.
\]
Therefore, $\mathcal{K}_{s,p}u = u + \curl\curl u.$
Hence, Theorem \ref{imPTH} ensures that $$\mathcal{K}_{s,p} : H_{0}^{s,p}(\curl;\Omega) \rightarrow (H_{0}^{-s,q}(\curl;\Omega))'$$
is an isomorphism, for all $(s,{1/p})\in R_{\Omega}.$ 

For all $t>0,$ we define the operator
\[
 \mathcal{B} : L_{s}^{p}(\Omega)\times L_{s}^{p}(\Omega) \rightarrow L_{s}^{p}(\Omega)\times L_{s}^{p}(\Omega)
\]
by 
\[
 \mathcal{B} \left
(\begin{array}{c} 
               A_1 \\ 
               A_2
     \end{array}
      \right)
: = \left
(\begin{array}{c} 
               A_1 - tk^2 A_1 \\ 
               A_2 - t a(\cdot)A_2
     \end{array}
      \right).
\]
Therefore, a simple calculation shows that
\[
 (\mathcal{L}_{-s,q}^{*}\mathcal{B}\mathcal{L}_{s,p}-\mathcal{K}_{s,p})u = -t[\curl(a(x)\curl u) + k^2u].
\]
\begin{lemma}\label{Lipschitz-continuity}
 The operator $\mathcal{B}$ is Lipschitz continuous for every fixed $(s,1/p) \in R_{\Omega}^{+}$.
\end{lemma}
\begin{proof}
\begin{align*}
&\|\mathcal{B}\left
(\begin{array}{c} 
               A_1 \\ 
               A_2
     \end{array}
      \right)
- \mathcal{B}\left
(\begin{array}{c} 
               B_1 \\ 
               B_2
     \end{array}
      \right)
\|_{L_{s}^{p}(\Omega)\times L_{s}^{p}(\Omega)} \\
 =& \|\left
(\begin{array}{c} 
               A_1 - tk^2A_1\\ 
               A_2 - ta(\cdot)A_2
     \end{array}
      \right)
- \left
(\begin{array}{c} 
               B_1 - tk^2B_1 \\ 
               B_2 - t a(\cdot) B_2
     \end{array}
      \right)
\|_{L_{s}^{p}(\Omega)\times L_{s}^{p}(\Omega)} \\
 =& \|\left
(\begin{array}{c} 
               (1-tk^2)(A_1-B_1)\\ 
               (1-ta(\cdot))(A_2-B_2)
     \end{array}
      \right)
\|_{L_{s}^{p}(\Omega)\times L_{s}^{p}(\Omega)} \\
 = & \left[\|(1-tk^2)(A_1-B_1)\|_{L_{s}^{p}(\Omega)}^{p} + \|(1-ta(x))(A_2-B_2)\|_{L_{s}^{p}(\Omega)}^{p}\right]^{{1/p}}\\
& (\text{Using Lemma \ref{KP}}) \\
 \leq & \left[|1-tk^2|^p\|(A_1-B_1)\|_{L_{s}^{p}(\Omega)}^{p}+ (C(s,p))^p\|1-ta(\cdot)\|_{W^{s,\infty}(\Omega)}^p\|(A_2-B_2)\|_{L_{s}^{p}(\Omega)}^{p}\}\right]^{{1/p}}.  \\
\end{align*}
We, now estimate the norm $\|1-ta(\cdot)\|_{W^{s,\infty}(\Omega)}$. Recall that, for $0<s<1$, we have, $$W^{s,\infty}(\Omega) = C^{0,s}(\overline\Omega)$$ with the norm
\[
\|\phi\|_{W^{s,\infty}(\Omega)}
= \|\phi\|_{L^{\infty}(\Omega)} + \sup_{\substack{
x, y\in\Omega \\
x\neq y}}
\frac{|\phi(x)-\phi(y)|}{|x-y|^s},
\]
where, $C^{0,s}(\overline\Omega)$ is the H{\"o}lder continuous space with exponent $s$. Hence
\[
\begin{split}
\|1-ta(\cdot)\|_{W^{s,\infty}(\Omega)} 
& =  \|1-ta(\cdot)\|_{L^{\infty}(\Omega)} + \sup_{\substack{
x, y\in\Omega \\
x\neq y}}
\frac{|(1-ta(x))-(1-ta(y)|}{|x-y|^s} \\
& =\sup_{x\in\Omega}|1-ta(x)| + t\sup_{\substack{
x, y\in\Omega \\
x\neq y}}
\frac{|a(x)-a(y)|}{|x-y|^s}.
\end{split}
\]
Since, $m\leq |a(x)|\leq M, \forall \ x\in\Omega$ and the H{\"o}lder semi-norm of $a$ is denoted by
\[
|a|_{C^{0,s}} :=\sup_{\substack{
x, y\in\Omega \\
x\neq y}}
\frac{|a(x)-a(y)|}{|x-y|^s},
 \]
then 
\[
\begin{split}
\|1-ta(\cdot)\|_{W^{s,\infty}(\Omega)}^p
& \leq [|1-tm| + t|a|_{C^{0,s}}]^p.
\end{split}
\]
Recall that, the coefficient $a$ is taken to be H{\"older} continuous with exponent $s$ in $\Omega$, i.e., there exists $\tilde M>0$ such that
$|a|_{C^{0,s}} \leq\tilde M$. Hence
\[
\begin{split}
\|\mathcal{B}\left
(\begin{array}{c} 
               A_1 \\ 
               A_2
     \end{array}
      \right)
- \mathcal{B}\left
(\begin{array}{c} 
               B_1 \\ 
               B_2
     \end{array}
      \right)
\|_{L_{s}^{p}(\Omega)\times L_{s}^{p}(\Omega)} 
\leq& \left[\max\{|1-tk^2|^p, (C(s,p))^p[|1-tm| + t{\tilde M}]^p\}\right]^{1/p} \\
 &\times \left[\|A_1-B_1\|_{L_{s}^{p}(\Omega)}^{p} + \|A_2 - B_2\|_{L_{s}^{p}(\Omega)}^{p}\right]^{1/p}.
\end{split}
\]
We set 
\begin{equation}\label{kappa-0}
k_0(s,p) := \max\{|1-tk^2|, C(s,p)[|1-tm| + t{\tilde M}]\},
\end{equation}
then we have
\begin{align*}
 \|\mathcal{B}\left
(\begin{array}{c} 
               A_1 \\ 
               A_2
     \end{array}
      \right)
- \mathcal{B}\left
(\begin{array}{c} 
               B_1 \\ 
               B_2
     \end{array}
      \right)
\|_{L_{s}^{p}(\Omega)\times L_{s}^{p}(\Omega)} 
& \leq k_0(s,p)
\| 
\left
(\begin{array}{c} 
               A_1 \\ 
               A_2
     \end{array}
      \right)
- \left
(\begin{array}{c} 
               B_1 \\ 
               B_2
     \end{array}
      \right)
\|_{L_{s}^{p}(\Omega)\times L_{s}^{p}(\Omega)}
\end{align*}
which means that $\mathcal{B}$ is Lipschitz with the Lipschitz constant $k_0(s,p)$.
\end{proof}
For all $(s,1/p)\in R_{\Omega},$ we define the operator $\mathcal{Q}_f$ as follows:
\[
\begin{split}
 \mathcal{Q}_f u : 
& = \mathcal{K}_{s,p}^{-1}\left(\mathcal{L}_{-s,q}^{*}\mathcal{B}\mathcal{L}_{s,p}u + tf\right)\\
& = u - t\mathcal{K}_{s,p}^{-1}\left[\{\curl(a(x)\curl u) + k^2u\}-f\right],
\end{split}
\]
where $u\in H_{0}^{s,p}(\curl;\Omega)$.
Our main aim is to show that $\mathcal{Q}_f$ is a contraction mapping, which is the key point to prove Theorem \ref{pert}.
\begin{notation}
For all $(s,1/p) \in R_{\Omega}$, we define $\mathcal{M}_{s,p}$ as follows
\begin{align*}
 \mathcal{M}_{s,p} 
  := \sup_{\substack{
u\in H_{0}^{s,p}(\curl;\Omega), \\
\|\mathcal{K}_{s,p}u\|_{(H_{0}^{-s,q}(\curl;\Omega))'}\leq 1
}}\|u\|_{H_{0}^{s,p}(\curl;\Omega)}.
\end{align*}
It is clear that
$
 \mathcal{M}_{s,p} = \|\mathcal{K}_{s,p}^{-1}\|_{(H_{0}^{-s,q}(\curl;\Omega))'\rightarrow H_{0}^{s,p}(\curl;\Omega)}.
$
\end{notation}
Let us first prove the following lemma.
\begin{lemma}\label{matrii}
 The operator $\mathcal{Q}_f : H_{0}^{s,p}(\curl;\Omega) \rightarrow H_{0}^{s,p}(\curl;\Omega)$ is Lipschitz with the Lipschitz constant $k_{0}(s,p) \mathcal{M}_{s,p}$ for all $(s,1/p) \in R_{\Omega}^{+}$.
\end{lemma}
\begin{proof}
\begin{align*}
&\|\mathcal{L}_{-s,q}^{*}\mathcal{B}\mathcal{L}_{s,p}\|_{(H_{0}^{-s,q}(\curl;\Omega))' \rightarrow H_{0}^{s,p}(\curl;\Omega)} \\
 =& \sup_{\|u\|_{H_{0}^{s,p}(\curl;\Omega)}\leq 1}\|\mathcal{L}_{-s,q}^{*}\mathcal{B}\mathcal{L}_{s,p}u\|_{(H_{0}^{-s,q}(\curl;\Omega))'}\\
 =& \sup_{\|u\|_{H_{0}^{s,p}(\curl;\Omega)}\leq 1}\sup_{\|v\|_{H_{0}^{-s,q}(\curl;\Omega)}\leq 1} {}_{(H_{0}^{-s,q}(\curl;\Omega))'}\langle \mathcal{L}_{-s,q}^{*}\mathcal{B}\mathcal{L}_{s,p}u,v\rangle_{H_{0}^{-s,q}(\curl;\Omega)}\\
 =& \sup_{\|u\|_{H_{0}^{s,p}(\curl;\Omega)}\leq 1}\sup_{\|v\|_{H_{0}^{-s,q}(\curl;\Omega)}\leq 1}{}_{L_{s}^{p}(\Omega)\times L_{s}^{p}(\Omega)}\langle \mathcal{B}\mathcal{L}_{s,p}u,\mathcal{L}_{-s,q}v\rangle_{L_{-s}^{q}(\Omega)\times L_{-s}^{q}(\Omega)}\\
 \leq & \sup_{\|u\|_{H_{0}^{s,p}(\curl;\Omega)}\leq 1}\sup_{\|v\|_{H_{0}^{-s,q}(\curl;\Omega)}\leq 1}\|\mathcal{B}\mathcal{L}_{s,p}u\|_{L_{s}^{p}(\Omega)\times L_{s}^{p}(\Omega)}\|\mathcal{L}_{-s,q}v\|_{L_{-s}^{q}(\Omega)\times L_{-s}^{q}(\Omega)}\\
 \leq & \sup_{\|u\|_{H_{0}^{s,p}(\curl;\Omega)}\leq 1}\|\mathcal{B}\mathcal{L}_{s,p}u\|_{L_{s}^{p}(\Omega)\times L_{s}^{p}(\Omega)}\; \text{(since $\mathcal{L}_{-s,q}$ is an isometry)}\ \\
 \leq & k_0(s, p)\sup_{\|u\|_{H_{0}^{s,p}(\curl;\Omega)}\leq 1} \Vert \mathcal{L}_{s,p}u \Vert_{H_{0}^{s,p}(\curl;\Omega)}\; \text{(Using Lemma \ref{Lipschitz-continuity})}\\
 \leq & k_{0}(s,p)\; \text{(since $\mathcal{L}_{s,p}$ is an isometry)}.
\end{align*}
Now
\[
\begin{split}
 \|\mathcal{K}_{s,p}^{-1}\mathcal{L}_{-s,q}^{*}\mathcal{B}\mathcal{L}_{s,p}\|_{(H_{0}^{-s,q}(\curl;\Omega))'\rightarrow H_{0}^{s,p}(\curl;\Omega)}
&\leq \|\mathcal{K}_{s,p}^{-1}\|\|\mathcal{L}_{-s,q}^{*}\mathcal{B}\mathcal{L}_{s,p}\|\\
&\leq \mathcal{M}_{s,p}k_0(s,p)
\end{split}
\]
and then, for $f\in (H_{0}^{-s,q}(\curl;\Omega))'$, we have
\[
\begin{split}
 \|\mathcal{Q}_fu-\mathcal{Q}_fv\|_{H_{0}^{s,p}(\curl;\Omega)}
& = \|\mathcal{K}_{s,p}^{-1}\mathcal{L}_{-s,q}^{*}\mathcal{B}\mathcal{L}_{s,p}(u-v)\|_{H_{0}^{s,p}(\curl;\Omega)}\\
& \leq \mathcal{M}_{s,p} k_0(s,p)\|u-v\|_{H_{0}^{s,p}(\curl;\Omega)}.
\end{split}
\]
Therefore, $\mathcal{Q}_f$ is a Lipschitz map from $H_{0}^{s,p}(\curl;\Omega)$
into itself with the Lipschitz constant $\mathcal{M}_{s,p} k_0(s,p)$ for $(s,1/p) \in R_{\Omega}^{+}$.
\end{proof}
For any fixed $(s_0,1/{p_0}), (s_1,1/{p_1})$ in $R_{\Omega}$,
we define the operators 
\[
\mathcal{P}_{s_0,p_0} : L_{s_0}^{p_0}(\Omega)\times L_{s_0}^{p_0}(\Omega) \rightarrow L_{s_0}^{p_0}(\Omega)\times L_{s_0}^{p_0}(\Omega)
\]
and
\[
\mathcal{P}_{s_1,p_1} : L_{s_1}^{p_1}(\Omega)\times L_{s_1}^{p_1}(\Omega) \rightarrow L_{s_1}^{p_1}(\Omega)\times L_{s_1}^{p_1}(\Omega)
\]
by
 $\mathcal{P}_{s_0,p_0} :=\mathcal{L}_{s_0,p_0}\mathcal{K}_{s_0,p_0}^{-1}\mathcal{L}_{-s_0,q_0}^{*}$
and $\mathcal{P}_{s_1,p_1} :=\mathcal{L}_{s_1,p_1}\mathcal{K}_{s_1,p_1}^{-1}\mathcal{L}_{-s_1,q_1}^{*}$, respectively. Observe that these operators are linear and bounded.

We state the following lemma, which is a consequence of the complex interpolation theorem.
\begin{lemma}\label{interfinn}
For any fixed $(s_0,1/{p_0}), (s_1,1/{p_1})\in R_{\Omega}$, the operators $\mathcal{P}_{s_0,p_0}$ and $\mathcal{P}_{s_1,p_1}$ are bounded. Then the operator 
\[
\mathcal{P}_{s,p} : L_{s}^{p}(\Omega)\times L_{s}^{p}(\Omega) \rightarrow L_{s}^{p}(\Omega)\times L_{s}^{p}(\Omega),
\]
defined by
\[
\mathcal{P}_{s,p} :=\mathcal{L}_{s,p}\mathcal{K}_{s,p}^{-1}\mathcal{L}_{-s,q}^{*},
\]
is bounded and satisfies the following estimate
\[
\|\mathcal{P}_{s,p}\| \leq \|\mathcal{P}_{s_0,p_0}\|^{1-\theta} \|\mathcal{P}_{s_1,p_1}\|^{\theta},
\]
for all $(s,1/p) \in R_{\Omega}$ and $s,p$ satisfying $s=(1-\theta)s_0+\theta s_1$,
$\frac{1}{p} = \frac{1-\theta}{p_0} + \frac{\theta}{p_1}$ with $0<\theta<1, s_0, s_1\in \mathbb{R}, s_0\neq s_1, 1<p_0, p_1<\infty$.
\end{lemma}
\begin{proof}
The proof follows from Theorem \ref{InTERp} and [\cite{Bergh}, Theorem 4.1.2].
\end{proof}
Recall that our goal is to show the operator $\mathcal{Q}_f$ is a contraction map. For that, it is enough to show the Lipschitz constant $\mathcal{M}_{s,p}k_0(s,p)$ 
is strictly less than $1$. 
In order to do that we state the following two lemmas.
\begin{lemma}\label{ESIviee}
The operator $\mathcal{K}_{s,p}$ is bounded and invertible for all $(s,1/p) \in R_{\Omega}$.
Moreover, for any $(s_0,1/{p_0}), (s_1,1/{p_1}) \in R_{\Omega}$, we have the following estimate
\begin{equation}\label{InESS}
 \mathcal{M}_{s,p} \leq \mathcal{M}_{s_0,p_0}^{1-\theta} \mathcal{M}_{s_1,p_1}^{\theta},
\end{equation}
for all $(s,1/p) \in R_{\Omega}$ and $s,p$ satisfying $s=(1-\theta)s_0+\theta s_1$,
$\frac{1}{p} = \frac{1-\theta}{p_0} + \frac{\theta}{p_1}$ with $0<\theta<1.$
\end{lemma}
\begin{proof}
The bounded invertibility of the operator $\mathcal{K}_{s,p}$ follows from Theorem \ref{imPTH}. It is enough to prove the estimate \eqref{InESS}.
Since the operators $\mathcal{P}_{s_0,p_0}$ and $\mathcal{P}_{s_1,p_1}$ are bounded and $\|\mathcal{L}_{s_0,p_0}\|=\|\mathcal{L}_{-s_0,q_0}\|=1$, then $\|\mathcal{P}_{s_0,p_0}\|\leq\mathcal{M}_{s_0,p_0}$ and $\|\mathcal{P}_{s_1,p_1}\|\leq\mathcal{M}_{s_1,p_1}$ for any fixed $(s_0,1/{p_0}), (s_1,1/{p_1})\in R_{\Omega}$.

Hence, applying Lemma \ref{interfinn}, we obtain that the operator 
\[
\mathcal{P}_{s,p} : L_{s}^{p}(\Omega)\times L_{s}^{p}(\Omega) \rightarrow L_{s}^{p}(\Omega)\times L_{s}^{p}(\Omega)
\]
is bounded with the estimate
\begin{align}
 \|\mathcal{P}_{s,p}\| 
& \leq \|\mathcal{P}_{s_0,p_0}\|^{1-\theta} \|\mathcal{P}_{s_1,p_1}\|^{\theta} \nonumber \\
&  \leq \mathcal{M}_{s_0,p_0}^{1-\theta} \mathcal{M}_{s_1,p_1}^{\theta}, \label{innnew}
\end{align}
for all $(s,1/p) \in R_{\Omega}$ and $s,p$ satisfying $s=(1-\theta)s_0+\theta s_1$,
$\frac{1}{p} = \frac{1-\theta}{p_0} + \frac{\theta}{p_1}$ with $0<\theta<1.$

Let us consider $f\in (H_{0}^{-s,q}(\curl;\Omega))'$, where $\frac{1}{p} + \frac{1}{q} = 1$.
Define a linear functional
\[
 \mathcal{Z} : Im(\mathcal{L}_{-s,q}) \subset L_{-s}^{q}(\Omega)\times L_{-s}^{q}(\Omega) \rightarrow \mathbb{R}
\]
by \[\langle \mathcal{Z}, \mathcal{L}_{-s,q}v\rangle : ={} _{(H_{0}^{-s,q}(\curl;\Omega))'}\langle f,v\rangle_{H_{0}^{-s,q}(\curl;\Omega)},\]
 for all $v \in H_{0}^{-s,q}(\curl;\Omega),$ where 
$$\mathcal{L}_{-s,q} : H_{0}^{-s,q}(\curl;\Omega) \rightarrow L_{-s}^{q}(\Omega)\times L_{-s}^{q}(\Omega).$$
Note that the above definition makes sense since $v$ is uniquely determined by $\mathcal{L}_{-s,q}.$
Now,
\[
 \begin{split}
  \|\mathcal{Z}\|
& = \sup_{
\substack{
v\in H_{0}^{-s,q}(\curl;\Omega)\\
v\neq 0
}} 
\frac{|\langle \mathcal{Z}, \mathcal{L}_{-s,q}v\rangle|}{\|\mathcal{L}_{-s,q}v\|_{L_{-s}^{q}(\Omega)\times {L_{-s}^{q}(\Omega)}}}\\
& = \sup_{
\substack{
v\in H_{0}^{-s,q}(\curl;\Omega)\\
v\neq 0
}} 
\frac{|\langle f, v\rangle|}{\|v\|_{H_{0}^{-s,q}(\curl;\Omega)}}\\
& = \|f\|_{(H_{0}^{-s,q}(\curl;\Omega))'}.
 \end{split}
\]
By Hahn-Banach extension theorem, $\mathcal{Z}$ can be extended to a continuous linear functional (again denoted by $\mathcal{Z}$) on 
$L_{-s}^{q}(\Omega)\times L_{-s}^{q}(\Omega)$ with the same norm $\|\mathcal{Z}\| = \|f\|_{(H_{0}^{-s,q}(\curl;\Omega))'}.$

Moreover, $\mathcal{L}_{-s,q}^{*}\mathcal{Z} = f$ because
\[
 \langle \mathcal{L}_{-s,q}^{*}\mathcal{Z}, v\rangle = \langle \mathcal{Z},\mathcal{L}_{-s,q}v\rangle = {}_{(H_{0}^{-s,q}(\curl;\Omega))'}\langle f,v\rangle_{H_{0}^{-s,q}(\curl;\Omega)}.
\]
Define, $u := \mathcal{K}_{s,p}^{-1} f$, where $f = \mathcal{L}_{-s,q}^{*}\mathcal{Z}$.
Therefore,
\[
 \begin{split}
  \mathcal{L}_{s,p}u 
&= \mathcal{L}_{s,p}\mathcal{K}_{s,p}^{-1}\mathcal{L}_{-s,q}^{*}\mathcal{Z}\\
& = \mathcal{P}_{s,p}\mathcal{Z}. 
 \end{split}
\]
Hence,
\[
\begin{split}
 \|u\|_{H_{0}^{s,p}(\curl;\Omega)}
& = \|\mathcal{L}_{s,p}u\|_{L_{s}^{p}(\Omega)\times L_{s}^{p}(\Omega)} \\
& = \|\mathcal{P}_{s,p}\mathcal{Z}\|_{L_{s}^{p}(\Omega)\times L_{s}^{p}(\Omega)}\\
& \leq \|\mathcal{P}_{s,p}\| \|\mathcal{Z}\|_{L_{s}^{p}(\Omega)\times L_{s}^{p}(\Omega)} \\
& (\text{using}\ \eqref{innnew}) \\
& \leq \mathcal{M}_{s_0,p_0}^{1-\theta} \mathcal{M}_{s_1,p_1}^{\theta} \|f\|_{(H_{0}^{-s,q}(\curl;\Omega))'}
\end{split}
\]
i.e., $\mathcal{M}_{s,p} = \sup_{\substack{\|f\|_{(H_{0}^{-s,q}(\curl;\Omega))'}\leq 1}}\|u\|_{H_{0}^{s,p}(\curl;\Omega)} \leq \mathcal{M}_{s_0,p_0}^{1-\theta} \mathcal{M}_{s_1,p_1}^{\theta}.$
\end{proof}
\begin{lemma}
We have 
\[
 \mathcal{M}_{0,2} = 1.
\]
\end{lemma}
\begin{proof}
 We prove this part in two steps.
\begin{Step}
In this step, we show that $\mathcal{M}_{0,2} \leq 1.$
\end{Step}
Recall that, $\mathcal{K}_{0,2}u^f := \curl\curl u^f + u^f = f$, where the operator 
$$\mathcal{K}_{0,2} : H_{0}(\curl;\Omega) \rightarrow (H_{0}(\curl;\Omega))',$$ 
is bounded and invertible, see Theorem \ref{imPTH}. Therefore,
\begin{equation}\label{m02}
 \begin{split}
  \mathcal{M}_{0,2}
 = \|\mathcal{K}_{0,2}^{-1}\| 
& = \sup_{\|f\|_{(H_{0}(\curl;\Omega))'}\leq 1}\|\mathcal{K}_{0,2}^{-1}f\|_{H_{0}(\curl;\Omega)}\\
% & = \sup_{\|g\|_{(H_{0}(\curl;\Omega))'}\leq 1}\sup_{\|f\|_{(H_{0}(\curl;\Omega))'}\leq 1} |{}_{H_{0}(\curl;\Omega)}\langle \mathcal{K}_{0,2}^{-1}f, g\rangle_{(H_{0}(\curl;\Omega))'}|\\
% & = \sup_{\|g\|_{(H_{0}(\curl;\Omega))'}\leq 1}\sup_{\|f\|_{(H_{0}(\curl;\Omega))'}\leq 1} |{}_{H_{0}(\curl;\Omega)}\langle u^f, g\rangle_{(H_{0}(\curl;\Omega))'}|\\
& =  \sup_{\|f\|_{(H_{0}(\curl;\Omega))'}\leq 1} \|u^f\|_{H_{0}(\curl;\Omega)}.
\end{split}
\end{equation}
Note that, $u^f$ satisfies the following Maxwell problem
\begin{equation}\label{maXwEL}
 \curl\curl u^f + u^f = f
\end{equation}
in the weak sense, i.e., in particular, we have
\[
 \int_{\Omega}|\curl u^f|^2 + \int_{\Omega}|u^f|^2 = \int_{\Omega}f\cdot u^f.
\]
Hence, using H{\"o}lder inequality, we have
\begin{equation*}
\begin{split} 
\|u^f\|_{H_0(\curl;\Omega)}^{2} 
& = \int_{\Omega}|\curl u^f|^2 + \int_{\Omega}|u^f|^2\\
& \leq \|f\|_{(H_0(\curl;\Omega))'}\|u^f\|_{H_0(\curl; \Omega)},
\end{split}
\end{equation*}
i.e., 
\begin{equation}\label{mo3}
 \|u^f\|_{H_0(\curl;\Omega)} \leq \|f\|_{(H_0(\curl;\Omega))'}.
\end{equation}
Combining \eqref{m02} and \eqref{mo3}, we deduce that $\mathcal{M}_{0,2} \leq 1.$
\begin{Step}
 In this step, we prove that, there exists $\tilde u \in H_0(\curl;\Omega)$ with $\|\mathcal{K}_{0,2}\tilde u\|_{(H_0(\curl;\Omega))'}\leq 1$
such that $\|\tilde u\|_{H_0(\curl;\Omega)}=1.$
\end{Step}
Take $u_0 \in H_0(\curl;\Omega)$ such that $\|u_0\|\neq 0.$ Define $\tilde u := \frac{u_0}{\|u_0\|}.$
Therefore, $\|\tilde u\|_{H_0(\curl;\Omega)}=1.$
Also,
 \begin{align*}
  \|\mathcal{K}_{0,2}\tilde u\|_{(H_0(\curl;\Omega))'} 
 =& \sup_{\|v\|_{H_0(\curl;\Omega)}\leq 1} |{}_{(H_0(\curl;\Omega))'}\langle \mathcal{K}_{0,2}\tilde u, v\rangle_{H_0(\curl;\Omega)}|\\
 =& \sup_{\|v\|_{H_0(\curl;\Omega)}\leq 1}\left[\int_{\Omega}\curl\tilde u\cdot \curl v + \int_{\Omega}\tilde u\cdot v \right]\\
 \leq &\sup_{\|v\|_{H_0(\curl;\Omega)}\leq 1}\left[\|\curl \tilde u\|_{L^2(\Omega)}\|\curl v\|_{L^2(\Omega)} + \|\tilde u\|_{L^2(\Omega)}\|v\|_{L^2(\Omega)}\right]\\
 \leq& \sup_{\|v\|_{H_0(\curl;\Omega)}\leq 1}(\frac{1}{2}(\|\curl v\|_{L^2(\Omega)}^{2} + \|v\|_{L^2(\Omega)}^{2}) \\
&+ \frac{1}{2}(\|\curl\tilde u\|_{L^2(\Omega)}^{2} + \|\tilde u\|_{L^2(\Omega)}^{2}) )\\
 \leq& \frac{1}{2} + \frac{1}{2} = 1,
 \end{align*}
i.e., $\mathcal{M}_{0,2} = 1.$
\end{proof}
Now, we are in a position to prove that $\mathcal{Q}_f$ is a contraction map.
\begin{proposition}\label{CONtracMaP}
 The operator $\mathcal{Q}_f : H_{0}^{s,p}(\curl;\Omega) \rightarrow H_{0}^{s,p}(\curl;\Omega)$
is a contraction map, for all $(s,1/p)\in S^{+},$ where $S^{+}$ is defined in Theorem \ref{pert}. 
\end{proposition}
\begin{proof}
From Lemma \ref{matrii} we have,
\[
 \|\mathcal{Q}_fu - \mathcal{Q}_fv\|_{H_{0}^{s,p}(\curl;\Omega)}
\leq \mathcal{M}_{s,p}k_{0}(s,p)\|u - v\|_{H_{0}^{s,p}(\curl;\Omega)}.
\]
To prove $\mathcal{Q}_f$ to be a contraction map, we need to show $\mathcal{M}_{s,p}k_{0}(s,p) <1, \ \forall \ (s,1/p) \in S^{+}.$ 
Now, fix any $(s_0,1/{p_0})\in R_{\Omega}^{+}$ and take a particular point $(0,1/2)$ in the region $R_{\Omega}^{+}$, then from Lemma \ref{ESIviee} we have the following estimate
\[
\mathcal{M}_{s,p} \leq \mathcal{M}_{s_0,p_0}^{1-\theta} \mathcal{M}_{0,2}^{\theta},
\]
where $s=(1-\theta)s_0, \frac{1}{p}=\frac{1-\theta}{p_0} + \frac{\theta}{2}$
and $0<\theta<1$. Therefore, $\mathcal{M}_{s,p}k_{0}(s,p) <1$ if we can show that
\begin{equation}\label{russss}
{\mathcal{M}}_{s_0,p_0}^{1-\theta} {\mathcal{M}}_{0,2}^{\theta} k_{0}(s,p) <1.
\end{equation}
Note that $\mathcal{M}_{0,2} =1,$ then passing $\log$ both sides of \eqref{russss} we have,
\begin{equation}\label{janeja}
(1-\theta)\log\mathcal{M}_{s_0,p_0} + \log k_0(s,p) < 0.
\end{equation}
Recall that
\begin{equation}\label{k-s-p}
k_0(s,p) =\max\{|1-tk^2|, C(s,p)[|1-tm| + t \tilde M]\}.
\end{equation}
We choose $t=\frac{m}{M^2}$. Then $k_0$ becomes
\[
k_0(s,p) =\max\left\{|1-\frac{mk^2}{M^2}|, C(s,p)\left[1-\frac{m^2}{M^2} + \frac{m\tilde M}{M^2}\right]\right\}.
\]
Now under the following conditions on $m,M,\tilde M, k^2>0$
\begin{equation}\label{condi1}
|1-\frac{mk^2}{M^2}| < 1
\end{equation}
and
\begin{equation}\label{condi2}
1-\frac{m^2}{M^2} + \frac{m\tilde M}{M^2} < \frac{1}{C(s,p)} 
\end{equation}
we obtain $k_0(s,p) <1$ for all $R_{\Omega}^{+}$.  
So, $\mathcal{M}_{s,p}k_{0}(s,p) <1$ if $(s,1/p)$ satisfies the following properties:
\begin{enumerate}
\item[(i)]
$(1-\theta)\log\mathcal{M}_{s_0,p_0} +\log k_0(s,p) < 0.$
\item[(ii)]
$s=(1-\theta)s_0.$
\item[(iii)]
$\frac{1}{p} = \frac{1-\theta}{p_0} + \frac{\theta}{2}$ with $0<\theta<1$,
\end{enumerate}
for $(s_0,1/{p_0})\in R_{\Omega}^{+}$ with these appropriate choice of $m,M,\tilde M,k^2$.

% In figure \ref{Fig 1}, we now give a geometrical interpretation of the above properties. We first fix $(s_0,1/{p_0}) \in R_{\Omega}$. 
% Then by interpolation we have that the points $(s,1/p)$ will be on the straight line joining $(0,1/2)$ and $(s_0,1/{p_0})$. 
% The above property $(i)$ says that, in the case of $a$ to be matrix measurable function, the points $(s,1/p)$ will be laying on some part of this straight line containing $(0,1/2)$. 
% Now, if we take any other point $(s_0,1/{p_0})\in R_{\Omega}$ and proceed the same argument, then we end up with the dashed region $S$ in the figure \ref{Fig 1}, where $\mathcal{M}_{s,p}k_{0}^{\frac{1}{p}} <1$. 
% Hence, $\mathcal{Q}_f$ is contraction for all $(s,1/p)\in S$.
\end{proof}

\subsection{End of the proof of Theorem \ref{pert}} We consider the three issues (uniqueness, existence and stability) separately.
\bigskip

{\textbf{Uniqueness}}

We start by proving the uniqueness of the solutions for the operator equation
\[
Au := \curl(a(x)\curl u) + k^2u = f.
\]
\begin{SSCase}
$(s,1/p) \in S \cap \{(s,1/p); p>2, s>0\}.$
\end{SSCase} 
In this range of $s$ and $p$ we have,
$H_{0}^{s,p}(\curl;\Omega) \subset H_0(\curl;\Omega)$.
Since, $$A : H_0(\curl;\Omega) \rightarrow (H_0(\curl;\Omega))'$$
is invertible then the fixed point $u$ of $\mathcal{Q}_f$ is the unique solution to $Au = f.$
\begin{SSCase}
$(s,1/p) \in S \cap \{(s,1/p); p>2, s>0\}^{c}.$
\end{SSCase}
For a given data $f$, let us consider $u_1$ and $u_2$ in $H_{0}^{s,p}(\curl;\Omega)$ be two solutions of the operator equation $Au=f$, where $(s,1/p) \in S \cap \{(s,1/p); p>2, s>0\}^{c}$,
i.e., $Au_1 = Au_2$. Since $A$ is linear then $A(u_1-u_2)=0$. Now, $0 \in (H_{0}^{-s,q}(\curl;\Omega))'$, for all $(s,1/p) \in S \cap \{(s,1/p); p>2, s>0\}$, 
then by applying Case 1, the operator equation $Au=0$ has $u:=0$ as the unique solution in $H_{0}^{s,p}(\curl;\Omega)$, so we have $u_1-u_2=0$. Hence, $A$ is injective. 

\bigskip
\textbf{Existence}
\begin{SSSCase}
$(s,1/p) \in S^{+}$.
\end{SSSCase}
Existence of the solution in $H_{0}^{s,p}(\curl;\Omega)$ of the operator equation
$\mathcal{Q}_fu = u$ is due to the fixed point theorem, as $\mathcal{Q}_f$ is a contraction map.
Hence, the fixed point $u\in H_{0}^{s,p}(\curl;\Omega)$ of $\mathcal{Q}_f$ is a
solution of 
\[
 Au := \curl(a(x)\curl u) + k^2 u = f,
\]
i.e., $$A : H_{0}^{s,p}(\curl;\Omega) \rightarrow (H_{0}^{-s,q}(\curl;\Omega))'$$ is onto, for all 
$(s,1/p)\in S^{+}$ and $\frac{1}{p} + \frac{1}{q} =1$.  
\begin{SSSCase}
$(s,1/p) \in S^{-}$.
\end{SSSCase}
Since the matrix $a$ is symmetric then $A=A^*$. Recall that the adjoint of an invertible operator is invertible. So, $$A^* : (H_{0}^{-s,q}(\curl;\Omega))^{''} \rightarrow (H_{0}^{s,p}(\curl;\Omega))'$$
is invertible for all $(s,1/p)\in S^{+}$. As $H_0^{s,p}(\curl;\Omega)$ is reflexive for all $(s,1/p)\in R_{\Omega}$, see Lemma \ref{refle}, therefore
\[
A : H_{0}^{-s,q}(\curl;\Omega) \rightarrow (H_{0}^{s,p}(\curl;\Omega))'
\]
is invertible for all $(s,1/p)\in S^{+}$,
i.e.
\[
A : H_{0}^{s,p}(\curl;\Omega) \rightarrow (H_{0}^{-s,q}(\curl;\Omega))'
\]
is invertible for all $(s,1/p)\in S^{-}$, recalling that
\[
S^{-}=\{(s,1/p)\in R_{\Omega}^{-} ; (-s,1/q)\in S^{+} \}.
\]

\bigskip

\textbf{Stability}
Finally, we finish the proof by deriving the stability estimate of the solution in terms of the given data.
\begin{SSSSCase}
 $(s,1/p) \in S^{+}$.
\end{SSSSCase}
If $f, g\in(H_{0}^{-s,q}(\curl;\Omega))'$ are given and $u,v$ are the fixed points of $\mathcal{Q}_f, \mathcal{Q}_g$
respectively. Then
\[
\begin{split}
& \|u-v\|_{H_{0}^{s,p}(\curl;\Omega)} \\
& = \|\mathcal{Q}_fu-\mathcal{Q}_gv\|_{H_{0}^{s,p}(\curl;\Omega)}\\
& \leq \mathcal{M}_{s,p}k_{0}(s,p)\|u-v\|_{H_{0}^{s,p}(\curl;\Omega)} + \mathcal{M}_{s,p}\frac{m}{M^2}\|f-g\|_{(H_{0}^{-s,q}(\curl;\Omega))'}
\end{split}
\]
i.e., 
\[
 \|u-v\|_{H_{0}^{s,p}(\curl;\Omega)} \leq \frac{m}{M^2}\mathcal{M}_{s,p}(1-\mathcal{M}_{s,p}k_{0}(s,p))^{-1}\|f-g\|_{(H_{0}^{-s,q}(\curl;\Omega))'}.
\]
Therefore, there exists $C :=  \frac{m}{M^2}\mathcal{M}_{s,p}(1-\mathcal{M}_{s,p}k_{0}(s,p))^{-1}>0,$ such that
\[
 \|u\|_{H_{0}^{s,p}(\curl;\Omega)}\leq C\|f\|_{(H_{0}^{-s,q}(\curl;\Omega))'}.
\]
\begin{SSSSCase}
$(s,1/p) \in S^{-}$.
\end{SSSSCase}
Note that the operator
\[
A : H_{0}^{s,p}(\curl;\Omega) \rightarrow (H_{0}^{-s,q}(\curl;\Omega))'
\]
is invertible and $A^{-1}$ is bounded for all $(s,1/p) \in S^{+}$. Hence by the open mapping theorem the operator $A$ is bounded. 
Now, from [\cite{rudin}, Theorem 4.15] and the reflexivity of the spaces $H_0^{s, p}(\curl;\Omega)$, we obtain that
\[
(A^{*})^{-1} : (H_0^{s,p}(\curl;\Omega))' \rightarrow H_0^{-s,q}(\curl;\Omega)
\]
is a bounded linear operator for all $(s,1/p) \in S^{+}$. Since $A=A^{*}$, then we have,
\[
\begin{split}
\|u\|_{H_0^{s,p}(\curl;\Omega)}
& = \|A^{-1}f\|_{H_0^{s,p}(\curl;\Omega)} \\
& \leq C \|f\|_{(H_0^{-s,q}(\curl;\Omega))'},
\end{split}
\]
for all $(s,1/p) \in S^{-}$, where $C>0$.

\section*{Acknowledgements}
MK was supported by the ERC Starting Grant and he is very thankful to Mikko Salo for his support. MS was partially supported by RICAM. 
The authors also would like to express their gratitude to the University of Jyv{\"a}skyl{\"a}, Finland and RICAM, Austrian Academy of Sciences, Austria, 
where most of the work has been done.

\bibliographystyle{abbrv}
\bibliography{manas}

\end{document}